\documentclass[11pt]{article}

\usepackage{amsmath, amsthm, amssymb}
\usepackage{mathtools}
\usepackage{enumitem}

\usepackage{tikz}
\usetikzlibrary{arrows}

\usepackage[colorlinks=true,linkcolor=blue,citecolor=blue,urlcolor=blue]{hyperref}

\numberwithin{equation}{section}

\newtheorem{theorem}{Theorem}[section]
\newtheorem{corollary}[theorem]{Corollary}
\newtheorem{lemma}[theorem]{Lemma}
\newtheorem{question}[theorem]{Question}

\newtheorem{conjecture}[theorem]{Conjecture}

\theoremstyle{definition}
\newtheorem{definition}[theorem]{Definition}

\newtheorem{obs}[theorem]{Observation}

\theoremstyle{remark}
\newtheorem{remark}[theorem]{Remark}

\newenvironment{funding}{\section*{Funding}}{}

\newcommand\Sc{{\mathcal{S}}}
\newcommand\Hc{{\mathcal{H}}}

\newcommand\Fc{{\mathcal{F}}}

\newcommand{\lw}{1.0}

\DeclarePairedDelimiter{\floor}{\lfloor}{\rfloor}

\newcommand\blfootnote[1]{%
  \begingroup
  \renewcommand\thefootnote{}\footnote{#1}%
  \addtocounter{footnote}{-1}%
  \endgroup
}

\begin{document}

\title{Three-chromatic geometric hypergraphs%
\blfootnote{The extended abstract of this paper appeared in the proceedings of SoCG~2022.}}

\author{
Gábor Damásdi\thanks{ELTE Eötvös Loránd University and Alfréd Rényi Institute of Mathematics, Budapest, Hungary.
\texttt{damasdigabor@caesar.elte.hu}}
\and
Dömötör Pálvölgyi\thanks{ELTE Eötvös Loránd University and Alfréd Rényi Institute of Mathematics, Budapest, Hungary.
\texttt{domotor.palvolgyi@ttk.elte.hu}}
}

\date{}

\maketitle

\begin{abstract}
We prove that for any planar convex body $C$ there is a positive integer $m$ with the property that any finite point set $P$ in the plane can be three-colored such that there is no translate of $C$ containing at least $m$ points of $P$, all of the same color. As a part of the proof, we show a strengthening of the Erdős-Sands-Sauer-Woodrow conjecture. Surprisingly, the proof also relies on the two dimensional case of the Illumination conjecture. 
\end{abstract}

\maketitle


\section{Introduction}


Our main result is the following.

\begin{theorem}\label{thm:general_three_col}
For any planar convex body $C$ there is a positive integer $m=m(C)$ such that any finite point set $P$ in the plane can be three-colored in a way that there is no translate of $C$ containing at least $m$ points of $P$, all of the same color.
\end{theorem}

This result closes a long line of research about coloring points with respect to planar range spaces that consist of translates of a fixed set, a problem that was initiated by Pach over forty years ago \cite{Pach80}.
In general, a pair $(P, \Sc)$, where $P$ is a set of points in the plane and $\Sc$ is a family of subsets of the plane, called the \emph{range space}, defines a \emph{primal} hypergraph $\Hc(P,\Sc)$ whose vertex set is $P$, and for each $S\in\Sc$ we add the edge $S\cap P$ to the hypergraph.
Given any hypergraph $\Hc$, a planar realization of $\Hc$ is defined as a pair $(P, \Sc)$ for which $\Hc(P,\Sc)$ is isomorphic to $\Hc$.
If $\Hc$ can be realized with some pair $(P, \Sc)$ where $\Sc$ is from some family $\Fc$, then we say that $\Hc$ is realizable with $\Fc$.
The hypergraph, where the elements of the range space $\Sc$ are the vertices and the points $P$ define the edges such that $\{S\in\Sc\mid p\in S\}$ is an edge for every $p\in P$, is known as the \emph{dual} hypergraph of $\Hc(P,\Sc)$ and is denoted by $\Hc(\Sc,P)$.
If $\Hc=\Hc(\Sc,P)$ where $\Sc$ is from some family $\Fc$, then we say that $\Hc$ has a dual realization with $\Fc$.
Pach observed \cite{Pach80,surveycd} that if $\Fc$ is the family of translates of some set, then $\Hc$ has a dual realization with $\Fc$ if and only if $\Hc$ has a (primal) realization with $\Fc$.

Pach proposed to study the chromatic number of hypergraphs realizable with different geometric families $\Fc$.
It is important to distinguish between two types of hypergraph colorings that we will use, the \emph{proper} coloring and the \emph{polychromatic} coloring. 

\begin{definition}
A hypergraph is \emph{properly $k$-colorable} if its vertices can be colored with $k$ colors so that each edge contains points from at least two color classes. Such a coloring is called a \emph{proper $k$-coloring}.
If a hypergraph has a proper $k$-coloring but not a proper $(k-1)$-coloring, then it is called \emph{$k$-chromatic}.

A hypergraph is \emph{polychromatic $k$-colorable} if its vertices can be colored with $k$ colors so that each edge contains points from each color class. Such a coloring is called a \emph{polychromatic $k$-coloring}.
\end{definition}

Note that for a polychromatic $k$-coloring to exist, it is necessary that each edge of the underlying hypergraph has at least $k$ vertices.
More generally, we say that a hypergraph is \emph{$m$-heavy} if each of its edges has at least $m$ vertices.

The main question that Pach raised can be rephrased for translates as follows.

\begin{question}
    For which planar families $\Fc$ is there an $m_k=m(\Fc,k)$ such that any $m_k$-heavy hypergraph realizable with $\Fc$ has a proper/polychromatic $k$-coloring? 
\end{question}

Initially, this question has been mainly studied for polychromatic $k$-colorings (known in case of a dual range space as \emph{cover-decomposition} problem), and it was shown that such an $m_k$ exists if $\Fc$ is the family of translates of some convex polygon \cite{Pach86,TT07,PT10}, or the family of all halfplanes \cite{wcf2,MR2844088}, or the homothetic\footnote{A \emph{homothetic copy}, or \emph{homothet}, is a scaled and translated (but non-rotated) copy of a set. We always require the scaling factor to be positive. Note that this is sometimes called a positive homothet.} copies of a triangle \cite{octants} or of a square \cite{homotsquare}, while it was also shown that even $m_2$ does not exist if $\Fc$ is the family of translates of some appropriate concave polygon \cite{MR2364757,MR2679054} or any body\footnote{By \emph{body}, we always mean a compact subset of the plane with a non-empty interior, though our results (and most of the results mentioned) also hold for sets that are unbounded, or that contain an arbitrary part of their boundary, and are thus neither open, nor closed. This is because a realization of a hypergraph can be perturbed slightly to move the points off from the boundaries of the sets realizing the respective edges of the hypergraph.} with a smooth boundary \cite{unsplittable}.
It was also shown that there is no $m_k$ for proper $k$-colorings if $\Fc$ is the family of all lines \cite{MR2364757} or all axis-parallel rectangles \cite{Chen}; for these families, the same holds in case of dual realizations \cite{MR2364757,PT08}.
For homothets of convex polygons other than triangles, it is known that there is no $m_2$ for dual realizations \cite{kovacs}, unlike for primal realizations.
Higher dimensional variants \cite{octants,CKMU13} and improved bounds for $m_k$ have been also studied \cite{Alou,MR2812512,MR3151767,MR3216669,MR3126347,CKMPUV20}.
For other results, see also the decade old survey \cite{surveycd}, or the up-to-date website \url{https://coge.elte.hu/cogezoo.html}.

If $\Fc$ is the translates or homothets of some planar convex body, it is an easy consequence of the properties of generalized Delaunay-triangulations and the
Four Color Theorem that any hypergraph realizable with $\Fc$ is proper 4-colorable if every edge
contains at least two vertices.
We have recently shown that this cannot be improved for homothets.

\begin{theorem}[Dam\'asdi, Pálvölgyi \cite{fourchromatic}]
 Let $C$ be any convex body in the plane that has two parallel supporting lines such that $C$ is strictly convex in some neighborhood of the two points of tangencies. For any positive integer $m$, there exists a 4-chromatic $m$-uniform hypergraph that is realizable with homothets of $C$.
\end{theorem}

For translates, we recall the following result.
 
\begin{theorem}[Pach, Pálvölgyi \cite{unsplittable}]\label{thm:unsplittable}
 Let $C$ be any convex body in the plane that has two\footnote{In the published version of our paper we have incorrectly claimed that we can prove that one such line is sufficient; however, that proof was incorrect.} parallel supporting lines such that $C$ is strictly convex in some neighborhood of the two points of tangencies. For any positive integer $m$, there exists a 3-chromatic $m$-uniform hypergraph that is realizable with translates of $C$.
\end{theorem}

This left only the following question open: Is it true for any planar convex body $C$ that there is a positive integer $m$ such that no 4-chromatic $m$-uniform hypergraph is realizable with translates of $C$?
Our Theorem \ref{thm:general_three_col} answers this question affirmatively for all $C$ by showing that all realizable $m$-heavy hypergraphs are three-colorable for some $m$.
This has been hitherto known to hold only when $C$ is a polygon (in which case 2 colors suffice \cite{PT10}, and 3 colors are known to be enough even for homothets \cite{3propercol}) and pseudodisk families that intersect in a common point \cite{MR4169259} (which generalizes the case when $C$ is unbounded, in which case 2 colors suffice \cite{unsplittable}).

The proof of Theorem \ref{thm:general_three_col} relies on a surprising connection with two other famous results, the solution of the two dimensional case of the Illumination conjecture \cite{MR76368}, and a recent solution of the Erdős-Sands-Sauer-Woodrow conjecture by Bousquet, Lochet and Thomassé~\cite{esswproof}. 
In fact, we need a generalization of the latter result, which we prove with the addition of one more trick to their method; this can be of independent interest.

Note that the extended abstract of our first proof attempt appeared recently in the proceedings of EuroComb 2021 \cite{threechromaticdisk}.
That proof did not use the above two results, however, it only worked when $C$ was a disk, and while the generalization to other convex bodies with a smooth boundary seemed feasible, we saw no way to extend it to arbitrary convex bodies.

\smallskip

The rest of the paper is organized as follows.\\
In Section \ref{sec:tools} we present the three main ingredients of our proof:
\begin{itemize}
    \item the Union Lemma (Section \ref{sec:unionlemma}),
    \item the Erdős-Sands-Sauer-Woodrow conjecture (Section \ref{sec:essw}),
    \item the Illumination conjecture (Section \ref{sec:illum}), which is a theorem of Levi in the plane.
\end{itemize}
In Section \ref{sec:proof} we give the detailed proof of Theorem \ref{thm:general_three_col}.\\
In the Section \ref{sec:esswproof} we prove our generalization of the Bousquet-Lochet-Thomassé theorem.\\
Finally, in Section \ref{sec:open}, we pose some problems left open.

\section{Tools}\label{sec:tools}
\subsection{Union Lemma}\label{sec:unionlemma}

Polychromatic colorability is a much stronger property than proper colorability. Any polychromatic $k$-colorable hypergraph is proper $2$-colorable. We generalize this trivial observation to the following statement about unions of polychromatic $k$-colorable hypergraphs.    
 
\begin{lemma}[Union Lemma]\label{lem:combine} Let $\Hc_1=(V,E_1),\dots, \Hc_{k-1}=(V,E_{k-1})$ be hypergraphs on a common vertex set $V$.  If $\Hc_1,\dots, \Hc_{k-1}$ are polychromatic $k$-colorable, then the hypergraph $\bigcup\limits_{i=1}^{k-1} \Hc_i=(V,\bigcup\limits_{i=1}^{k-1} E_i)$ is proper $k$-colorable.
\end{lemma}

\begin{proof}
Let $c_i:V\rightarrow \{1,\dots,k \}$ be the polychromatic coloring of the $i$-th hypergraph. Using the $c_i$-s we construct a proper coloring $c:V\rightarrow \{1,\dots,k \}$ for the union.
 Choose $c(v)\in \{1,\ldots,k\}$ such that it differs from each $c_i(v)$.
 We claim that $c$ is a proper $k$-coloring of $\bigcup\limits_{i=1}^{k-1} \Hc_i$.
 To prove this, it is enough to show that for every edge $H\in\Hc_i$ and for every color $j\in\{1,\ldots,k\}$, there is a $v\in H$ such that $c(v)\ne j$. 
 We can pick $v\in H$ for which $c_i(v)=j$.
 This finishes the proof.
\end{proof}
  
Lemma \ref{lem:combine} is sharp in the sense that for every $k$ there are $k-1$ hypergraphs such that each is polychromatic $k$-colorable, but their union is not properly $(k-1)$-colorable. Take for example the following $k-1$ hypergraphs. The vertex set $V$ of each hypergraph is the grid $ \{1,\dots,k \}^{k-1}$ and for each $i\in  \{1,\dots,k-1 \}$ the edge set $E_i$ is those subsets of $V$  where each one of $1, \dots, k$ appears in the $i$-th coordinate for some element of the subset. Clearly, $(V,E_i)$ is polychromatic $k$-colorable: simply color the vertices according to their $i$-th coordinate.  Suppose that their union has a proper $(k-1)$-coloring. The $i$-th color class cannot form an edge in $E_i$ since that would be a monochromatic edge. Therefore, there must be a number $l_i\in  \{1,\dots,k-1 \}$ such that no vertex of the $i$-th color has $l_i$ in its $i$-th coordinate. But this is a contradiction, as the vertex $(l_1,\dots, l_{k-1})$ does not have a color. This shows the sharpness of the lemma.\\

We will apply the Union Lemma combined with the theorem below.
A \emph{pseudoline arrangement} is a collection of simple curves, each of which splits $\mathbb R^2$ into two unbounded parts, such that any two curves intersect at most once.
A \emph{pseudohalfplane} is the region on one side of a pseudoline in such an arrangement.
For hypergraphs realizible by pseudohalfplanes the following was proved, generalizing a result of Smorodinsky and Yuditsky \cite{MR2844088} about halfplanes. 

\begin{theorem}[Keszegh-P\'alv\"olgyi \cite{abafree}]\label{thm:pseudohalfplane}
 Any $(2k-1)$-heavy hypergraph realizable by pseudohalfplanes is polychromatic $k$-colorable, i.e., given a finite set of points and a pseudohalfplane arrangement in the plane, the points can be $k$-colored such that every pseudohalfplane that contains at least $2k-1$ points contains all $k$ colors.
\end{theorem}

Combining Theorem \ref{thm:pseudohalfplane} with Lemma \ref{lem:combine} for $k=3$, we obtain the following.

\begin{corollary}\label{cor:pseudohalfplane}
 Any $5$-heavy hypergraph realizable by two pseudohalfplane families is proper $3$-colorable, i.e., given a finite set of points and two different pseudohalfplane arrangements in the plane, the points can be $3$-colored such that every pseudohalfplane that contains at least $5$ points contains two differently colored points.
\end{corollary}

\subsection{Erdős-Sands-Sauer-Woodrow conjecture}\label{sec:essw}
Given a quasi-order\footnote{A quasi-order $\prec$ is a reflexive and transitive relation, but it is not required to be antisymmetric, so $p\prec q\prec p$ is allowed, unlike for partial orders.} $\prec$ on a set $V$, we interpret it as a digraph $D=(V,A)$, where the vertex set is $V$ and a pair $(x,y)$ defines an arc in $A$ if $x \prec y$. 
The \emph{closed in-neighborhood} of a vertex $x\in V$ is $N^-(x)=\{x\}\cup \{y|(y,x)\in A \}$. Similarly the \emph{closed out-neighborhood} of a vertex $x$ is $N^+(x)=\{x\}\cup \{y|(x,y)\in A \}$. We extend this to subsets $S\subset V$ as $N^-(S) = \bigcup\limits_{ x\in S } N^-(x)$ and  $N^+(S) = \bigcup\limits_{ x\in S } N^+(x)$.
A set of vertices $S$ such that $N^+(S) = V$ is said to be \emph{dominating}. 
For $A,B\subset V$ we will also say that \emph{$A$ dominates $B$ } if $B\subset N^+(A)$.

A \emph{complete multidigraph} is a digraph where parallel edges are allowed and in which there is at least one arc between each pair of distinct vertices. Let
$D$ be a complete multidigraph whose arcs are the disjoint union of $k$ quasi-orders $\prec_1, \dots , \prec_k$ (parallel arcs are allowed). Define $N^-_i(x)$ (resp.\ $N^+_i(x)$) as the closed in-neighborhood (resp.\ out-neighborhood) of the digraph induced by $\prec_i$. 

Proving the conjecture of Erdős, and of Sands, Sauer and Woodrow \cite{sandssauer}, Bousquet, Lochet and Thomassé recently showed the following. 

\begin{theorem}[Bousquet, Lochet, Thomassé~\cite{esswproof}]\label{thm:multi_essw_old}
For every $k$, there exists an integer $f(k)$ such that if $D$ is a complete multidigraph whose arcs are the union of $k$ quasi-orders, then $D$ has a dominating set of size at most $f(k)$.
\end{theorem}

We show the following generalization of Theorem \ref{thm:multi_essw_old}. 

\begin{theorem}\label{thm:multi_essw_new}
For every pair of positive integers $k$ and $l$, there exist an integer $f(k,l)$ such that if $D=(V,A)$ is a complete multidigraph whose arcs are the union of $k$ quasi-orders $\prec_1,\dots, \prec_k$, then $V$ contains a family of pairwise disjoint subsets $S_{i}^j$  for $i\in [k]$, $j\in [l]$  with the following properties:

\begin{itemize}
    \item $|\bigcup\limits_{i,j}S_{i}^j|\le f(k,l)$
    \item For each vertex $v\in V\setminus \bigcup\limits_{i,j}S_{i}^j$ there is an $i\in [k]$ such that for each $j\in [l]$ there is an edge of $\prec_i$ from a vertex of $S_{i}^j$ to  $v$.
\end{itemize}
\end{theorem}

Note that disjointness is the real difficulty here, without it the theorem would trivially hold from repeated applications of Theorem \ref{thm:multi_essw_old}. 
We saw no way to derive Theorem \ref{thm:multi_essw_new} from Theorem \ref{thm:multi_essw_old}, but with an extra modification the proof goes through.
The proof of Theorem \ref{thm:multi_essw_new} can be found in Section \ref{sec:esswproof}.

\subsection{Hadwiger's Illumination conjecture and pseudolines}\label{sec:illum}

 
    Hadwiger's Illumination conjecture has a number of equivalent formulations and names\footnote{These include names such as Levi–Hadwiger Conjecture, Gohberg–Markus Covering Conjecture, Hadwiger Covering Conjecture, Boltyanski–Hadwiger Illumination Conjecture.}.  For a recent survey, see \cite{MR3816868}. We will use the following version of the conjecture.
    
    Let $\mathbb{S}^{d-1}$ denote the unit sphere in $\mathbb R^d$.
    For a convex body $C$, let $\partial C$ denote the boundary of $C$ and let $int(C)$ denote its interior.
    A direction (light) $u\in \mathbb{S}^{d-1}$ \emph{illuminates} $b\in \partial C$ if $\{b+\lambda u:\lambda>0 \}\cap int (C)\ne \emptyset$.  

    \begin{conjecture}
     The boundary of any convex body in $\mathbb{R}^d$ can be illuminated by $2^d$ or fewer directions.  Furthermore, the $2^d$ lights are necessary if and only if the body is a parallelepiped.
    \end{conjecture}
    
The conjecture is open in general. The $d=2$ case was settled in affirmative by Levi \cite{MR76368} in 1955. For $d=3$
the best result is due to Prymak \cite{Prymak}, who showed that 16 lights are enough, improving the earlier method of Papadoperakis \cite{MR1689273} with the help of a computer program.

In the following part we make an interesting connection between the Illumination conjecture  for $d=2$ and pseudolines. Roughly speaking, we show that the Illumination conjecture implies that for any convex body in the plane the boundary can be broken into three parts such that the translates of each part behave similarly to pseudolines, i.e., we get three pseudoline arrangements from the translates of the three parts. 

To put this into precise terms, we need some technical definitions and statements.
Fix a body $C$ and an injective parametrization of $\partial C$, $\gamma:[0,1]\rightarrow \partial C$, that follows $\partial C$ counterclockwise.
For each point $p$ of $\partial C$ there is a set of possible tangents touching at $p$. Let  $g(p)\subset \mathbb{S}^1$ denote the Gauss image of $p$, i.e., $g(p)$ is the set of unit outernormals of the tangent lines touching at $p$. Note that $g(p)$ is an arc of $\mathbb{S}^1$ and $g(p)$ is a proper subset of $\mathbb{S}^1$. 

Let $g_+:\partial C\rightarrow\mathbb{S}^1$ be the function that assigns to $p$ the counterclockwise last element of $g(p)$. (See Figure \ref{fig:gauss_tan} left.) Similarly let $g_-$ be the function that assigns to $p$ the clockwise last element of $g(p)$. Thus, $g(p)$ is the arc of $\mathbb{S}^1$ from $g_-(p)$ to $g_+(p)$. Let $|g(p)|$ denote the length of $g(p)$. 

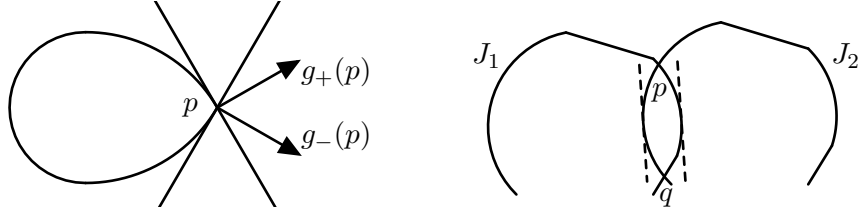
\begin{figure}[!ht]
    \centering
    \begin{tikzpicture}[line cap=round,line join=round,>=triangle 45,x=1.0cm,y=1.0cm]
\clip(5.824158146704215,1.6939909822621093) rectangle (10.8,4.4022061705050906);
\draw [shift={(7.,2.)},line width=\lw]  plot[domain=0.5235987755982986:1.5707963267948966,variable=\t]({1.*2.*cos(\t r)+0.*2.*sin(\t r)},{0.*2.*cos(\t r)+1.*2.*sin(\t r)});
\draw [shift={(7.,4.)},line width=\lw]  plot[domain=4.71238898038469:5.759586531581288,variable=\t]({1.*2.*cos(\t r)+0.*2.*sin(\t r)},{0.*2.*cos(\t r)+1.*2.*sin(\t r)});
\draw [shift={(7.,3.)},line width=\lw]  plot[domain=1.5707963267948966:4.71238898038469,variable=\t]({1.*1.*cos(\t r)+0.*1.*sin(\t r)},{0.*1.*cos(\t r)+1.*1.*sin(\t r)});
\draw [line width=\lw,domain=5.824158146704215:10.274171485061972] plot(\x,{(--18.124355652982153-1.7320508075688785*\x)/1.});
\draw [line width=\lw,domain=5.824158146704215:10.274171485061972] plot(\x,{(-12.12435565298215--1.7320508075688785*\x)/1.});
\draw [->,line width=1.pt] (8.732050807568879,3.) -- (9.832456454322593,3.6353194963710407);
\draw [->,line width=1.pt] (8.732050807568879,3.) -- (9.844045808452828,2.3579893869021333);
\draw (8.15,3.25) node[anchor=north west] {$p$};
\draw (9.7,2.95) node[anchor=north west] {$g_-(p)$};
\draw (9.7,3.8) node[anchor=north west] {$g_+(p)$};
\end{tikzpicture}
~~~~~~~~
\begin{tikzpicture}[line cap=round,line join=round,>=triangle 45,x=0.6cm,y=0.6cm]
\clip(1.010405779360095,0.7779725023481115) rectangle (9.945938145792228,4.969084292744436);
\draw (4.75,3.7) node[anchor=north west] {$p$};
\draw (4.9,1.4) node[anchor=north west] {$q$};
\draw [line width=\lw,dash pattern=on 3pt off 3pt] (5.52700500582115,3.969615338937309)-- (5.701093102682993,1.3323544637799518);
\draw [line width=\lw,dash pattern=on 3pt off 3pt] (4.691398477047223,3.8957218914504184)-- (4.86141088560568,1.3202036253261604);
\draw [line width=\lw] (5.516440449569205,1.8865088084802843)-- (4.995922845299596,1.0326578708773406);
\draw [line width=\lw] (8.392512263686044,1.2568699665550929)-- (8.913029867955656,2.1107209041580446);
\draw [line width=\lw] (3.0820370497082816,4.613210333135324)-- (4.995922845299596,4.032657870877341);
\draw [line width=\lw] (8.39251226368604,4.256869966555094)-- (6.478626468094716,4.837422428813083);
\draw [shift={(3.495922845299593,2.532657870877341)},line width=\lw]  plot[domain=-0.30951591373703113:0.7853981633974475,variable=\t]({1.*2.1213203435596446*cos(\t r)+0.*2.1213203435596446*sin(\t r)},{0.*2.1213203435596446*cos(\t r)+1.*2.1213203435596446*sin(\t r)});
\draw [shift={(3.495922845299593,2.532657870877341)},line width=\lw]  plot[domain=1.7671635199760698:3.9269908169872414,variable=\t]({1.*2.121320343559645*cos(\t r)+0.*2.121320343559645*sin(\t r)},{0.*2.121320343559645*cos(\t r)+1.*2.121320343559645*sin(\t r)});
\draw [shift={(6.892512263686043,2.756869966555093)},line width=\lw]  plot[domain=-0.3095159137370276:0.7853981633974495,variable=\t]({1.*2.121320343559643*cos(\t r)+0.*2.121320343559643*sin(\t r)},{0.*2.121320343559643*cos(\t r)+1.*2.121320343559643*sin(\t r)});
\draw [shift={(6.892512263686043,2.756869966555093)},line width=\lw]  plot[domain=1.7671635199760762:3.9269908169872414,variable=\t]({1.*2.1213203435596544*cos(\t r)+0.*2.1213203435596544*sin(\t r)},{0.*2.1213203435596544*cos(\t r)+1.*2.1213203435596544*sin(\t r)});
\draw (8.676673546001679,4.6) node[anchor=north west] {$J_2$};
\draw (0.8,4.6) node[anchor=north west] {$J_1$};
\end{tikzpicture}
\caption{Extremal tangents at a boundary point (on the left) and parallel tangents on two intersecting translates (on the right).}
    \label{fig:gauss_tan}
\end{figure}

\begin{obs}\label{obs:continuity}
$g_+\circ \gamma$ is continuous from the right and $g_-\circ \gamma$ is continuous from the left. 
\end{obs}


For $t_1<t_2$ let $\gamma_{[t_1,t_2]}$ denote the restriction of $\gamma$ to the interval $[t_1,t_2]$. For $t_1>t_2$ let $\gamma_{[t_1,t_2]}$ denote the concatenation of $\gamma_{[t_1,1]}$ and $\gamma_{[0,t_2]}$.
When it leads to no confusion, we identify $\gamma_{[t_1,t_2]}$ with its image, which is a closed connected part of the boundary $\partial C$. 
For such a $J=\gamma_{[t_1,t_2]}$, let $g(J)=\bigcup\limits_{p\in J}g(p)$. Clearly, $g(J)$ is an arc of $\mathbb{S}^1$ from $g_-(t_1)$ to $g_+(t_2)$; let $|g(J)|$ denote the length of this arc.

\begin{lemma}
 Let $C$ be a convex body and assume that $J$ is a closed connected part of $\partial C$ such that $|g(J)|<\pi$. Then there are no two translates of $J$ that intersect in more than one point.    
\end{lemma}

\begin{proof}
Suppose $J$ has two translates $J_1$ and $J_2$ such that they intersect in two points, $p$ and $q$. Now both $J_1$ and $J_2$ have a tangent that is parallel to the segment $pq$, but since they lie on different sides of the $pq$ line, they have opposite outer normal vectors. (See Figure \ref{fig:gauss_tan} right.) This shows that $J$ has two different tangents parallel to $pq$ and therefore $|g(J)|\ge \pi$. 
\end{proof}

\begin{lemma}\label{lemma:our_illumination}
  For a convex body $C$, which is not a parallelogram, and an injective parametrization $\gamma$ of $\partial C$, we can pick  $0\le t_1<t_2<t_3\le 1$ such that  $|g(\gamma_{[t_1,t_2]})|,|g(\gamma_{[t_2,t_3]})|$ and $|g(\gamma_{[t_3,t_1]})|$ are each strictly smaller than $\pi$.
\end{lemma}
    
\begin{proof}
    We use the 2-dimensional case of the Illumination conjecture (proved by Levi \cite{MR76368}). If $C$ is not a parallelogram, we can pick three directions, $u_1,u_2$ and $u_3$, that illuminate $C$. Pick $t_1$ such that $\gamma(t_1)$ is illuminated by both $u_1$ and $u_2$. To see why this is possible, suppose that the parts illuminated by $u_1$ and $u_2$ are disjoint. Each light illuminates a continuous open ended part of the boundary. So in this case there are two disjoint parts of the boundary that are not illuminated. If $u_3$ illuminates both, then it illuminates everything that is illuminated by $u_1$ or everything that is illuminated by $u_2$.     
    This would mean that two lights illuminate the whole boundary but this is not possible for any convex body. Indeed, suppose that two lights $u$ and $v$ illuminate the whole body. Then there is a halfplane $H$ through the origin that contains both vectors $u$ and $v$. Take a translate of $H$ that touches $C$. Clearly the touching point is not illuminated by either $u$ or $v$, a contradiction.
    
    Using the same argument, pick $t_2$ and $t_3$ such that $\gamma(t_2)$ is illuminated by both $u_2$ and $u_3$ and $\gamma(t_3)$ is illuminated by both $u_3$ and $u_1$. 
    
    Note that $u_1$ illuminates exactly those points for which $g_+(p)<u_1+\pi/2$ and $g_-(p)>u_1-\pi/2$.  Therefore, $|g(\gamma_{[t_1,t_3]})|<u_1+\pi/2-(u_1-\pi/2)=\pi$. Similarly $|g(\gamma_{[t_1,t_2]})|<\pi$ and $|g(\gamma_{[t_2,t_3]})|<\pi$.
\end{proof}
 
 Observation \ref{obs:continuity} and Lemma \ref{lemma:our_illumination} immediately imply the following statement.

\begin{lemma}\label{lemma:our_illumination_epsilon}
 For a convex body $C$, which is not a parallelogram, and an injective parametrization $\gamma$ of $\partial C$, we can pick  $0\le t_1<t_2<t_3\le 1$ and $\varepsilon>0$ such that  $|g(\gamma_{[t_1-\varepsilon,t_2+\varepsilon]})|$, $|g(\gamma_{[t_2-\varepsilon,t_3+\varepsilon]})|$ and $|g(\gamma_{[t_3-\varepsilon,t_1+\varepsilon]})|$ are each strictly smaller than $\pi$.
  \end{lemma}

\section{Proof of Theorem \ref{thm:general_three_col}}\label{sec:proof}
\subsection{Quasi-orders on planar point sets}

Cones provide a natural way to define quasi-orders on  point sets (see \cite{TT07} for an example where this idea was used). A \emph{cone} is a closed region in the plane that is bounded by two rays that emanate from the origin. For a cone $K$ let $-K$ denote the cone that is the reflection of $K$ across the origin and let $q+K$ denote the translate of $K$ by the vector $q$.  

\begin{obs}\label{obs:cones}
For any $p,q\in \mathbb{R}^2$ and cone $K$, the following are equivalent (see Fig.~\ref{fig:basic_cones}):
\begin{itemize}
    \item $p\in q+K$
    \item $q \in p+(-K)$
    \item $p+K\subseteq q+K$
\end{itemize}
\end{obs}

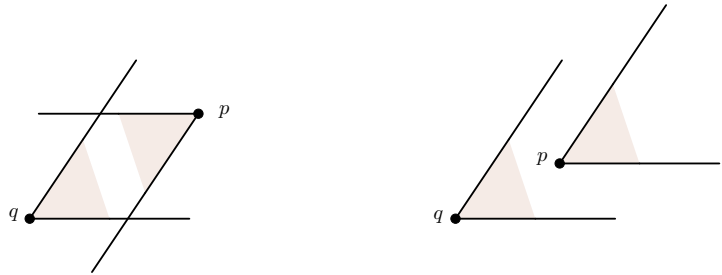
\begin{figure}[!ht]
    \centering
    \definecolor{zzttqq}{rgb}{0.6,0.2,0.}
    \scalebox{0.7}{
\begin{tikzpicture}[line cap=round,line join=round,>=triangle 45,x=0.5cm,y=0.5cm]
\clip(0.7905167827637798,-0.6536209763473118) rectangle (28.955457063301157,10.602349828485595);
\fill[line width=\lw,color=zzttqq,fill=zzttqq,fill opacity=0.10000000149011612] (21.914471376040254,4.093030278329455) -- (23.914471376040254,7.093030278329451) -- (24.914471376040254,4.093030278329455) -- cycle;
\fill[line width=\lw,color=zzttqq,fill=zzttqq,fill opacity=0.10000000149011612] (18.,2.) -- (20.,5.) -- (21.,2.) -- cycle;
\fill[line width=\lw,color=zzttqq,fill=zzttqq,fill opacity=0.10000000149011612] (2.,2.) -- (4.,5.) -- (5.,2.) -- cycle;
\fill[line width=\lw,color=zzttqq,fill=zzttqq,fill opacity=0.10000000149011612] (8.339322178085318,5.977538969999748) -- (6.339322178085318,2.977538969999748) -- (5.339322178085318,5.977538969999748) -- cycle;
\draw [line width=\lw] (21.914471376040254,4.093030278329455)-- (23.914471376040254,7.093030278329451);
\draw [line width=\lw] (23.914471376040254,7.093030278329451)-- (25.914471376040254,10.093030278329444);
\draw [line width=\lw] (21.914471376040254,4.093030278329455)-- (24.914471376040254,4.093030278329455);
\draw [line width=\lw] (24.914471376040254,4.093030278329455)-- (27.914471376040247,4.093030278329455);
\draw [line width=\lw] (18.,2.)-- (20.,5.);
\draw [line width=\lw] (20.,5.)-- (22.,8.);
\draw [line width=\lw] (18.,2.)-- (21.,2.);
\draw [line width=\lw] (21.,2.)-- (24.,2.);
\draw (16.899061667902384,2.598103922826639) node[anchor=north west] {$q$};
\draw (20.801131546911115,4.799271546882852) node[anchor=north west] {$p$};
\draw [line width=\lw] (2.,2.)-- (4.,5.);
\draw [line width=\lw] (4.,5.)-- (6.,8.);
\draw [line width=\lw] (2.,2.)-- (5.,2.);
\draw [line width=\lw] (5.,2.)-- (8.,2.);
\draw [line width=\lw] (8.339322178085318,5.977538969999748)-- (6.339322178085318,2.977538969999748);
\draw [line width=\lw] (6.339322178085318,2.977538969999748)-- (4.339322178085318,-0.022461030000251903);
\draw [line width=\lw] (8.339322178085318,5.977538969999748)-- (5.339322178085318,5.977538969999748);
\draw [line width=\lw] (5.339322178085318,5.977538969999748)-- (2.339322178085318,5.977538969999748);
\draw (8.844789225333082,6.550200338745749) node[anchor=north west] {$p$};
\draw (0.9405963934948848,2.6481304597370077) node[anchor=north west] {$q$};
\begin{scriptsize}
\draw [fill=black] (21.914471376040254,4.093030278329455) circle (2.5pt);
\draw [fill=black] (18.,2.) circle (2.5pt);
\draw [fill=black] (2.,2.) circle (2.5pt);
\draw [fill=black] (8.339322178085318,5.977538969999748) circle (2.5pt);
\end{scriptsize}
\end{tikzpicture}}\caption{Basic properties of cones.}
    \label{fig:basic_cones}
\end{figure}

For a cone $K$ let $\prec_K$ denote the relation on the points of the plane where a point $p$ is bigger than a point $q$ if and only if $p+K$ contains $q$. By Observation \ref{obs:cones}, this relation is transitive so it is a quasi-order. Recall that when $\prec_K$ is interpreted as a digraph, $qp$ is an edge if and only if $q \prec_K p$.

\begin{figure}[!ht]
    \centering
    \definecolor{qqttcc}{rgb}{0.,0.2,0.8}
\definecolor{yqqqqq}{rgb}{0.5019607843137255,0.,0.}
\definecolor{qqwuqq}{rgb}{0.,0.39215686274509803,0.}
\definecolor{qqttzz}{rgb}{0.,0.2,0.6}
\scalebox{0.7}{
\begin{tikzpicture}[line cap=round,line join=round,>=triangle 45,x=0.9cm,y=0.9cm]
\clip(2.8081291197505673,2.5852872443375357) rectangle (19.480030573405248,6.726256287901889);
\draw [shift={(14.,6.)},line width=\lw,color=qqttzz,fill=qqttzz,fill opacity=1.0] (0,0) -- (-135.:0.5401263969866527) arc (-135.:-71.56505117707799:0.5401263969866527) -- cycle;
\draw [shift={(15.,3.)},line width=\lw,color=qqwuqq,fill=qqwuqq,fill opacity=1.0] (0,0) -- (108.43494882292202:0.5401263969866527) arc (108.43494882292202:161.56505117707798:0.5401263969866527) -- cycle;
\draw [shift={(12.,4.)},line width=\lw,color=yqqqqq,fill=yqqqqq,fill opacity=1.0] (0,0) -- (-18.43494882292201:0.5401263969866527) arc (-18.43494882292201:45.:0.5401263969866527) -- cycle;
\draw [shift={(3.,4.)},line width=\lw,color=yqqqqq,fill=yqqqqq,fill opacity=1.0] (0,0) -- (-18.43494882292201:0.5401263969866527) arc (-18.43494882292201:45.:0.5401263969866527) -- cycle;
\draw [line width=\lw] (12.,4.)-- (14.,6.);
\draw [line width=\lw] (14.,6.)-- (15.,3.);
\draw [line width=\lw] (15.,3.)-- (12.,4.);
\draw [line width=\lw,color=qqttcc] (16.4144675126188,5.222364386041088)-- (16.,4.);
\draw [line width=\lw,color=qqttcc] (17.39966534681431,6.059782545107268)-- (17.8,4.2);
\draw [line width=\lw,color=qqttcc] (17.39966534681431,6.059782545107268)-- (17.,3.);
\draw [line width=\lw,color=qqttcc] (18.2124535600256,5.444033898735077)-- (17.8,4.2);
\draw [line width=\lw,color=qqttzz] (17.8,4.2)-- (17.,3.);
\draw [line width=\lw,color=qqwuqq] (19.185336421293663,3.732252661820378)-- (17.39966534681431,6.059782545107268);
\draw [line width=\lw,color=qqwuqq] (19.185336421293663,3.732252661820378)-- (18.2124535600256,5.444033898735077);
\draw [line width=\lw,color=qqwuqq] (17.8,4.2)-- (16.4144675126188,5.222364386041088);
\draw [line width=\lw,color=qqwuqq] (17.,3.)-- (16.,4.);
\draw [line width=\lw,color=yqqqqq] (16.4144675126188,5.222364386041088)-- (17.39966534681431,6.059782545107268);
\draw [line width=\lw,color=yqqqqq] (16.,4.)-- (18.2124535600256,5.444033898735077);
\draw [line width=\lw,color=yqqqqq] (17.8,4.2)-- (19.185336421293663,3.732252661820378);
\draw [line width=\lw,color=yqqqqq] (17.,3.)-- (19.185336421293663,3.732252661820378);
\draw [line width=\lw,color=yqqqqq] (16.,4.)-- (17.8,4.2);
\draw [line width=\lw,color=qqwuqq] (18.2124535600256,5.444033898735077)-- (17.39966534681431,6.059782545107268);
\draw [line width=\lw,color=qqwuqq] (16.4144675126188,5.222364386041088)-- (19.185336421293663,3.732252661820378);
\draw [line width=\lw,color=qqttcc] (17.,3.)-- (18.2124535600256,5.444033898735077);
\draw [line width=\lw,color=yqqqqq] (16.4144675126188,5.222364386041088)-- (18.2124535600256,5.444033898735077);
\draw [line width=\lw,color=qqttcc] (16.,4.)-- (17.39966534681431,6.059782545107268);
\draw [line width=\lw,color=yqqqqq] (16.,4.)-- (19.185336421293663,3.732252661820378);
\draw [line width=\lw,color=qqttcc] (17.,3.)-- (16.4144675126188,5.222364386041088);
\draw [line width=\lw] (3.,4.)-- (5.,6.);
\draw [line width=\lw] (6.,3.)-- (3.,4.);
\draw [line width=\lw,color=yqqqqq] (8.399665346814308,6.059782545107264)-- (7.414467512618802,5.222364386041083);
\draw [line width=\lw,color=yqqqqq] (7.804180545110583,5.553620463659098) -- (7.802122827418463,5.764536527101339);
\draw [line width=\lw,color=yqqqqq] (7.804180545110583,5.553620463659098) -- (8.012010032014645,5.517610404047007);
\draw [line width=\lw,color=yqqqqq] (9.212453560025601,5.444033898735072)-- (7.414467512618802,5.222364386041083);
\draw [line width=\lw,color=yqqqqq] (8.17944361442798,5.316676508181941) -- (8.293633375274837,5.494019448661143);
\draw [line width=\lw,color=yqqqqq] (8.17944361442798,5.316676508181941) -- (8.333287697369565,5.172378836115012);
\draw [line width=\lw,color=yqqqqq] (8.8,4.2)-- (7.,4.);
\draw [line width=\lw,color=yqqqqq] (7.765794289841775,4.085088254426863) -- (7.882105905312236,4.26104685218987);
\draw [line width=\lw,color=yqqqqq] (7.765794289841775,4.085088254426863) -- (7.917894094687763,3.938953147810129);
\draw [line width=\lw,color=yqqqqq] (10.185336421293664,3.732252661820378)-- (8.8,4.2);
\draw [line width=\lw,color=yqqqqq] (9.36473230652311,4.009322826499032) -- (9.544504005353444,4.119649415858658);
\draw [line width=\lw,color=yqqqqq] (9.36473230652311,4.009322826499032) -- (9.440832415940221,3.81260324596172);
\draw [line width=\lw,color=yqqqqq] (10.185336421293664,3.732252661820378)-- (7.,4.);
\draw [line width=\lw,color=yqqqqq] (8.458111127749135,3.8774366906380453) -- (8.606240642320259,4.027594830387427);
\draw [line width=\lw,color=yqqqqq] (8.458111127749135,3.8774366906380453) -- (8.579095778973405,3.704657831432951);
\draw [line width=\lw,color=yqqqqq] (10.185336421293664,3.732252661820378)-- (8.,3.);
\draw [line width=\lw,color=yqqqqq] (8.96463306203684,3.323224891359415) -- (9.041186483185905,3.5197685092421795);
\draw [line width=\lw,color=yqqqqq] (8.96463306203684,3.323224891359415) -- (9.144149938107761,3.2124841525781975);
\begin{scriptsize}
\draw [fill=black] (16.4144675126188,5.222364386041088) circle (2.5pt);
\draw [fill=black] (16.,4.) circle (2.5pt);
\draw [fill=black] (17.,3.) circle (2.5pt);
\draw [fill=black] (19.185336421293663,3.732252661820378) circle (2.5pt);
\draw [fill=black] (17.8,4.2) circle (2.5pt);
\draw [fill=black] (18.2124535600256,5.444033898735077) circle (2.5pt);
\draw [fill=black] (17.39966534681431,6.059782545107268) circle (2.5pt);
\draw [fill=black] (7.414467512618802,5.222364386041083) circle (2.5pt);
\draw [fill=black] (7.,4.) circle (2.5pt);
\draw [fill=black] (8.,3.) circle (2.5pt);
\draw [fill=black] (10.185336421293664,3.732252661820378) circle (2.5pt);
\draw [fill=black] (8.8,4.2) circle (2.5pt);
\draw [fill=black] (9.212453560025601,5.444033898735072) circle (2.5pt);
\draw [fill=black] (8.399665346814308,6.059782545107264) circle (2.5pt);
\end{scriptsize}
\end{tikzpicture}}\caption{Quasi-order on a point set.}
    \label{fig:ordering}
\end{figure}
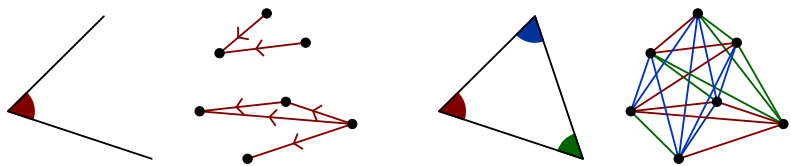

Suppose the cones $K_1, K_2, K_3$ are the translates of the three corners of a triangle so that all their apexes are in the origin, in other words the cones $K_1,-K_3,K_2,-K_1,K_3,-K_2$ partition the plane around the origin in this order. Then we will say that $K_1, K_2, K_3$  is a \emph{set of tri-partition} cones. In this case the intersection of any translates of $K_1, K_2, K_3$ forms a (possibly degenerate) triangle.

\begin{obs} 
Let $K_1,K_2,K_3$ be a set of tri-partition cones and let $P$ be a planar point set. Then any two distinct points of $P$ are comparable in either $\prec_{K_1}$, $\prec_{K_2}$ or $\prec_{K_3}$. (See Figure \ref{fig:ordering}.)
\end{obs}

In other words, when interpreted as digraphs, the union of $\prec_{K_1}$, $\prec_{K_2}$ and $\prec_{K_3}$ forms a complete multidigraph on $P$. As a warm up for the proof of Theorem \ref{thm:general_three_col}, we show the following theorem.

\begin{theorem}\label{thm:three_cones}
There exists a positive integer $m$ such that for any point set $P$, and any set of tri-partition cones $K_1,K_2,K_3$, we can three-color $P$ such that no translate of $K_1$, $K_2$ or $K_3$ that contains at least $m$ points of $P$ is monochromatic. 
\end{theorem}

\begin{proof}
We set $m$ to be $f(3,2)+13$  with the function of Theorem \ref{thm:multi_essw_new}. 
Consider the three quasi-orders $\prec_{K_1}$, $\prec_{K_2}$ or $\prec_{K_3}$.  Their union gives a complete multidigraph on $P$, hence we can apply Theorem \ref{thm:multi_essw_new} with $k=3$ and $l=2$, resulting in subsets $S_i^j$ for $i\in[3],j\in [2]$. Let $S=\bigcup\limits_{i\in [3],j\in[2]}S_i^j$. For each point $p\in P\setminus S$ there is an $i$ such that $\prec_{K_i}$ has an edge from a vertex of $S_{i,1}$ and $S_{i,2}$ to $p$. Let $P_1,P_2,P_3$ be the partition of $P\setminus S$ according to this $i$ value. 

We start by coloring the points of $S$. Color the points of $S_{1,1}\cup S_{2,1} \cup S_{3,1}$ with the first color and color the points of $S_{1,2}\cup S_{2,2}\cup S_{3,2}$ with the second color.

Any translate of $K_1$, $K_2$ or $K_3$ that contains $f(3,2)+13$ points of $P$, must contain $5$ points from either $P_1,P_2$ or $P_3$ by the pigeonhole principle. (Note that the cone might contain all points of $S$.) Therefore, it is enough to show that for each $i\in [3]$ the points of $P_i$ can be three-colored such that no translate of $K_1$, $K_2$, or $K_3$ that contains at least $5$ points of $P_i$ is monochromatic. 

Consider $P_1$; the proof is the same for $P_2$ and $P_3$. Take a translate of $K_1$ and suppose that it contains a point $p$ of $P_1$. By Theorem \ref{thm:multi_essw_new}, there is an edge of $\prec_{K_1}$ from a vertex of $S_{1,1}$ to $p$ and another edge from a vertex of $S_{1,2}$ to $p$. Thus any such translate contains a point from $S_{1,1}$ and another point from $S_{1,2}$, and hence it cannot be monochromatic. 

Therefore, we only have to consider the translates of $K_2$ and $K_3$. Two translates of a cone intersect at most once on their boundary. Hence, the translates of $K_2$ form a pseudohalfplane arrangement, and so do the translates of $K_3$. Therefore, by Corollary \ref{cor:pseudohalfplane}, there is a proper three-coloring for the translates of $K_2$ and $K_3$ together. 
\end{proof}

\begin{remark}
    From Theorem \ref{thm:three_cones}, it follows using standard methods (see Section \ref{sec:proofend}) that Theorem \ref{thm:general_three_col} holds for triangles. 
    This was of course known before, even for two-colorings of homothetic copies of triangles.
    Our proof cannot be modified for homothets, but a two-coloring would follow if instead of Corollary \ref{cor:pseudohalfplane} we applied a more careful analysis for the two cones.
\end{remark}

\subsection{Proof of Theorem \ref{thm:general_three_col}}\label{sec:proofend}

%
%
%
%
%
%
%

If $C$ is a parallelogram, then our proof method fails.
Luckily, translates of parallelograms (and other symmetric polygons) were the first for which it was shown that even two colors are enough \cite{Pach86}; in fact, by now we know that two colors are enough even for homothets of parallelograms \cite{homotsquare}.
So from now on we assume that $C$ is not a parallelogram. 


The proof of Theorem \ref{thm:general_three_col} relies on the same ideas as we used for Theorem \ref{thm:three_cones}. We partition $P$ into several parts, and for each part $P_i$, we divide the translates of $C$ into three families such that two of the families each form a pseudohalfplane arrangement over $P_i$, while the third family will only contain translates that are automatically non-monochromatic. Then Corollary \ref{cor:pseudohalfplane} gives us a proper three-coloring. As in the proof of Theorem \ref{thm:three_cones}, this is not done directly. First, we divide the plane using a grid, and then in each small square we will use Theorem \ref{thm:multi_essw_new} to discard some of the translates of $C$ at the cost of a bounded number of points.

\smallskip
Now we start the proof of Theorem \ref{thm:general_three_col}.
The first step is a classic divide and conquer idea \cite{Pach86}. We chose a constant $r=r(C)$ depending only on $C$ and divide the plane into a grid of squares of side length $r$. Since each translate of $C$ intersects some bounded number of squares, by the pigeonhole principle we can find for any positive integer $m$ another integer $m'$ such that the following holds: each translate $\hat C$ 
of $C$ that contains at least $m'$ points intersects a square $Q$ such that $\hat C\cap Q$ contains at least $m$ points.
For example, we can choose $m'=m(diam(C)/r+2)^2$, where $diam(C)$ denotes the diameter of $C$.
Therefore, it is enough to show the following localized version of Theorem \ref{thm:general_three_col}, since applying it separately for the points in each square of the grid provides a proper three-coloring of the whole point set.

\begin{theorem}\label{thm:local_three_col}
There is a positive integer $m$ such that for any convex body $C$ there is a positive real $r$ such that any finite point set $P$ in the plane that lies in a square of side length $r$ can be three-colored in a way that there is no translate of $C$ containing at least $m$ points of $P$, all of the same color.
\end{theorem}

We will show that $m$ can be chosen to be $f(3,2)+13$ with the function of Theorem \ref{thm:multi_essw_new}, independently of $C$.

\begin{proof}
We pick $r$ the following way. First we fix an injective parametrization $\gamma$ of $\partial C$ and then fix $t_1,t_2,t_3$ and $\varepsilon$ according to Lemma \ref{lemma:our_illumination_epsilon}. Let $\ell_1,\ell_2,\ell_3$ be the tangents of $C$ touching at $\gamma(t_1),\gamma(t_2)$ and $\gamma(t_3)$. Let $K_{1,2}$, $K_{2,3}$, $K_{3,1}$ be the set of tri-partition cones bordered by $\ell_1,\ell_2,\ell_3$, such that $K_{i,i+1}$ is bordered by $\ell_i$ on its counterclockwise side, and by $\ell_{i+1}$ on its clockwise side (see Figure \ref{fig:cone_in_C} left, and note that we always treat $3+1$ as 1 in the subscript). 

For a translate $\hat{C}$ of $C$ we will denote by $\hat{\gamma}$ the translated parametrization of $\partial \hat{C}$, i.e., $\hat{\gamma}(t)=\gamma(t)+v$ if $\hat{C}$ was translated by vector $v$.  Our aim is to choose $r$ small enough to satisfy the following two properties for each $i\in [3]$.

\begin{enumerate}[label=(\Alph*)]
    \item Let $\hat C$ be a translate of $C$, and $Q$ be a square of side length $r$ such that $\partial \hat C\cap Q\subset \hat{\gamma}_{[t_i+\varepsilon/2,t_{i+1}-\varepsilon/2]}$ (see Figure \ref{fig:cone_in_C} right). Then for any translate $K$ of $K_{i,i+1}$ whose apex is in $Q\cap \hat C$, we have $K\cap Q\subset \hat  C$. (I.e., $r$ is small with respect to $C$.)
    \item  Let $\hat  C$ be a translate of $C$, and $Q$ be a square of side length $r$ such that $\hat{\gamma}_{[t_i-\varepsilon/2,t_{i+1}+\varepsilon/2]}$ intersects $Q$. Then $\partial \hat C\cap Q\subset \hat{\gamma}_{[t_i-\varepsilon,t_{i+1}+\varepsilon]}$. (I.e., $r$ is small compared to $\varepsilon$.)
\end{enumerate}

\begin{figure}[!ht]
    \centering
  \definecolor{zzttqq}{rgb}{0.6,0.2,0.}
  \definecolor{uuuuuu}{rgb}{0.26666666666666666,0.26666666666666666,0.26666666666666666}

\begin{tikzpicture}[line cap=round,line join=round,>=triangle 45,x=0.5cm,y=0.5cm]
\clip(2.1640924696902344,-3.291941380065454) rectangle (16.64624177595318,6.606761736183365);
\fill[line width=\lw,fill=black,fill opacity=0.25] (3.768827753322032,4.392669166650977) -- (4.123243589875512,3.4575812484314694) -- (4.758793204447675,4.5339787755868235) -- cycle;
\fill[line width=\lw,fill=black,fill opacity=0.25] (14.058734708892569,5.861470654848949) -- (13.068769257766968,5.720161045913108) -- (13.374479808664983,5.132227738178157) -- cycle;
\fill[line width=\lw,fill=black,fill opacity=0.25] (6.332889037089297,-2.3723296870708763) -- (7.017143937316888,-1.6430867704000818) -- (5.978473200535815,-1.4372417688513646) -- cycle;
\draw [shift={(7.958515351695592,2.108914472950761)},line width=\lw]  plot[domain=2.6956780077804776:4.321854967035546,variable=\t]({1.*3.1083274241025274*cos(\t r)+0.*3.1083274241025274*sin(\t r)},{0.*3.1083274241025274*cos(\t r)+1.*3.1083274241025274*sin(\t r)});
\draw [shift={(7.261346221122771,2.5938329918446867)},line width=\lw]  plot[domain=0.13035761915140343:2.755875028289039,variable=\t]({1.*2.2743120841793814*cos(\t r)+0.*2.2743120841793814*sin(\t r)},{0.*2.2743120841793814*cos(\t r)+1.*2.2743120841793814*sin(\t r)});
\draw [shift={(6.496593035223344,2.298949087855251)},line width=\lw]  plot[domain=-1.4801162709845777:0.19311405339801058,variable=\t]({1.*3.0769654110024027*cos(\t r)+0.*3.0769654110024027*sin(\t r)},{0.*3.0769654110024027*cos(\t r)+1.*3.0769654110024027*sin(\t r)});
\draw [line width=\lw,domain=2.1640924696902344:16.64624177595318] plot(\x,{(--12.223776958212898--0.4526542136088514*\x)/3.17113631658728});
\draw [line width=\lw,domain=2.1640924696902344:16.64624177595318] plot(\x,{(--18.39532881276564-3.3853951579956414*\x)/1.2831281782249193});
\draw [line width=\lw,domain=2.1640924696902344:16.64624177595318] plot(\x,{(-21.960768293888048--2.565850114616926*\x)/2.407559228948587});
\draw (9.4,6.6) node[anchor=north west] {$\ell_1$};
\draw (4.6,-0.1) node[anchor=north west] {$\ell_2$};
\draw (6.6,2.6) node[anchor=north west] {$C$};
\draw (10.94582130433904,2.4) node[anchor=north west] {$\ell_3$};
\draw [shift={(3.768827753322032,4.392669166650977)},line width=\lw,fill=black,fill opacity=0.25]  plot[domain=-1.2085070485393068:0.14178417369315438,variable=\t]({1.*1.*cos(\t r)+0.*1.*sin(\t r)},{0.*1.*cos(\t r)+1.*1.*sin(\t r)});
\draw [shift={(6.332889037089297,-2.3723296870708763)},line width=\lw,fill=black,fill opacity=0.25]  plot[domain=0.817214862644781:1.9330856050504859,variable=\t]({1.*1.*cos(\t r)+0.*1.*sin(\t r)},{0.*1.*cos(\t r)+1.*1.*sin(\t r)});
\draw [shift={(14.058734708892569,5.861470654848949)},line width=\lw,fill=black,fill opacity=0.4000000059604645]  plot[domain=3.283376827282948:3.958807516234576,variable=\t]({1.*1.*cos(\t r)+0.*1.*sin(\t r)},{0.*1.*cos(\t r)+1.*1.*sin(\t r)});
\draw (13.6,5.5) node[anchor=north west] {$K_{3,1}$};
\draw (2.203321433221302,4.026165982141844) node[anchor=north west] {$K_{1,2}$};
\draw (6.7,-1.9) node[anchor=north west] {$K_{2,3}$};
\begin{scriptsize}
\draw [fill=uuuuuu] (6.939964069909312,4.845323380259829) circle (2.0pt);
\draw [fill=uuuuuu] (8.740448266037884,0.19352042754604953) circle (2.0pt);
\draw [fill=uuuuuu] (5.051955931546951,1.0072740086553358) circle (2.0pt);
\end{scriptsize}
\end{tikzpicture}
\begin{tikzpicture}[line cap=round,line join=round,>=triangle 45,x=0.8cm,y=0.8cm]
\clip(-0.5212593802625312,0.9024160297185335) rectangle (7.098126520651556,7.480250043437565);
\fill[line width=\lw,fill=black,fill opacity=0.30000001192092896] (2.9139611807128176,4.440100887949994) -- (3.068078600743505,3.0862602098415906) -- (4.272853164612676,4.54034726918462) -- cycle;
\fill[line width=\lw,color=zzttqq,fill=zzttqq,fill opacity=0.10000000149011612] (2.12382,3.74) -- (3.54248,3.74) -- (3.54248,5.15866) -- (2.12382,5.15866) -- cycle;
\draw [shift={(4.663491963072474,3.1523141871657336)},line width=\lw]  plot[domain=2.63100772848181:3.9408911121618377,variable=\t]({1.*2.759430143068236*cos(\t r)+0.*2.759430143068236*sin(\t r)},{0.*2.759430143068236*cos(\t r)+1.*2.759430143068236*sin(\t r)});
\draw [shift={(4.858950201988104,2.01321086543119)},line width=\lw]  plot[domain=1.0014831356942346:2.3788491897615827,variable=\t]({1.*3.6008052563532615*cos(\t r)+0.*3.6008052563532615*sin(\t r)},{0.*3.6008052563532615*cos(\t r)+1.*3.6008052563532615*sin(\t r)});
\draw [line width=\lw,domain=-0.8212593802625312:7.098126520651556] plot(\x,{(--5.714971243081739-2.5455417413714536*\x)/0.2897773217368045});
\draw [line width=\lw,domain=-0.8212593802625312:7.098126520651556] plot(\x,{(--15.596206877223619--0.21851199420715073*\x)/2.962044052434135});
\draw [shift={(2.9139611807128176,4.440100887949994)},line width=\lw,fill=black,fill opacity=0.30000001192092896]  plot[domain=-1.4574470824511945:0.07363728921063928,variable=\t]({1.*1.3625845885147592*cos(\t r)+0.*1.3625845885147592*sin(\t r)},{0.*1.3625845885147592*cos(\t r)+1.*1.3625845885147592*sin(\t r)});
\draw [line width=\lw] (2.9139611807128176,4.440100887949994)-- (3.068078600743505,3.0862602098415906);
\draw [line width=\lw] (4.272853164612676,4.54034726918462)-- (2.9139611807128176,4.440100887949994);
\draw [line width=\lw] (3.1733300410036582,2.161681526748409)-- (3.068078600743505,3.0862602098415906);
\draw [line width=\lw] (4.272853164612676,4.54034726918462)-- (5.154440683911443,4.6053825770666235);
\draw (4.3,5.9) node[anchor=north west,rotate=50] {$\hat{\gamma}(t_1)$};
\draw (0.6,3.183055506852521) node[anchor=north west] {$\hat{\gamma}(t_2)$};
\draw (6.487561621261355,6.433415706547415) node[anchor=north west] {$\ell_1$};
\draw (1.3,1.6925588406940881) node[anchor=north west] {$\ell_2$};
\draw [line width=\lw,color=zzttqq] (2.12382,3.74)-- (3.54248,3.74);
\draw [line width=\lw,color=zzttqq] (3.54248,3.74)-- (3.54248,5.15866);
\draw [line width=\lw,color=zzttqq] (3.54248,5.15866)-- (2.12382,5.15866);
\draw [line width=\lw,color=zzttqq] (2.12382,5.15866)-- (2.12382,3.74);
\draw (-0.7,3.7) node[anchor=north west] {$\hat{\gamma}(t_2-\varepsilon/2)$};
\draw (3.6,5.9) node[anchor=north west,rotate=50] {$\hat{\gamma}(t_1+\varepsilon/2)$};
\draw (4.35,3.919324944352469) node[anchor=north west] {$K$};
\begin{scriptsize}
\draw [fill=black] (1.9217694985931228,2.840204202956866) circle (2.0pt);
\draw [fill=black] (4.594036229290453,5.60425793853547) circle (2.0pt);
\draw [fill=black] (1.9127284810063392,3.370843316127941) circle (2.0pt);
\draw [fill=black] (4.030305286656739,5.517372120064994) circle (2.0pt);
\end{scriptsize}
\end{tikzpicture}

    \caption{Selecting the cones (on the left) and Property (A) (on the right).}
    \label{fig:cone_in_C}
\end{figure}
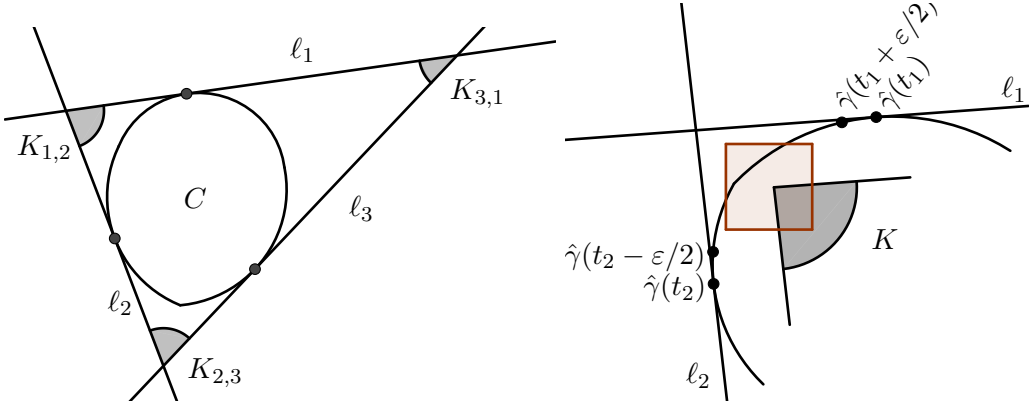

We show that an $r$ satisfying properties (A) and (B) can be found for $i=1$. The argument is the same for $i=2$ and $i=3$, and we can take the smallest among the three resulting values of $r$. 
 
First, consider property (A). Since the sides of $K$ are parallel to $\ell_1$ and $\ell_2$, the portion of $K$ that lies ``above'' the segment $\overline{\hat{\gamma}(t_1)\hat{\gamma}(t_2)}$ is in $\hat{C}$. Hence, if we choose $r$ small enough so that $Q$ cannot intersect $\overline{\hat{\gamma}(t_1)\hat{\gamma}(t_2)}$, then property (A) is satisfied.
We can choose $r$ to be smaller than $\frac{1}{\sqrt{2}}$ times the distance of the segments $\overline{\hat{\gamma}(t_1)\hat{\gamma}(t_2)}$ and $\overline{\hat{\gamma}(t_1+\varepsilon/2)\hat{\gamma}(t_2-\varepsilon/2)}$. 



Using that $\gamma$ is a continuous function on a compact set, we can pick $r$ such that property (B) is satisfied.
Therefore, there is an $r$ satisfying properties (A) and (B).


\bigskip
The next step is a subdivision of the point set $P$ using Theorem \ref{thm:multi_essw_new}, like we did in the proof of Theorem \ref{thm:three_cones}.
The beginning of our argument is exactly the same.

Apply Theorem \ref{thm:multi_essw_new} for the graph given by the union of $\prec_{K_{1,2}}$, $\prec_{K_{2,3}}$ and $\prec_{K_{3,1}}$. By Observation \ref{obs:cones}, this is indeed a complete multidigraph on $P$. 

We apply Theorem \ref{thm:multi_essw_new} with $k=3$ and $l=2$, resulting in subsets $S_i^j$ for $i\in[3],j\in [2]$. Let $S=\bigcup\limits_{i\in [3],j\in[2]}S_i^j$. For each point $p\in P\setminus S$ there is an $i$ such that $\prec_{K_{i,i+1}}$ has an edge from a vertex of $S_{i,1}$ and $S_{i,2}$ to $p$. Let $P_1,P_2,P_3$ be the partition of $P\setminus S$ according to this $i$ value. 

We start by coloring the points of $S$. Color the points of $S_{1,1}\cup S_{2,1} \cup S_{3,1}$ with the first color and color the points of $S_{1,2}\cup S_{2,2}\cup S_{3,2}$ with the second color.

Note that $m$ is at least $f(3,2)+13$. Any translate of $C$  that contains $f(3,2)+13$ points of $P$ must contain $5$ points from either $P_1,P_2$ or $P_3$. (Note that the cone might contain all points of $S$). Thus, it is enough to show that for each $i\in [3]$ the points of $P_i$ can be 3-colored so that no translate of $C$ that contains at least $5$ points of $P_i$ is monochromatic.

\smallskip

Consider $P_1$, the proof is the same for $P_2$ and $P_3$. We divide the translates of $C$ that intersect $Q$ into four (not necessarily disjoint) groups. Let $\mathcal{C}_0$ denote the translates where $\hat{C}\cap Q=Q$. Let $\mathcal{C}_1$ denote the translates for which $\partial \hat{C}\cap Q\subset \hat{\gamma}_{[t_1+\varepsilon/2,t_{2}-\varepsilon/2]}$.  Let $\mathcal{C}_2$ denote the translates for which $\partial \hat{C}\cap Q\cap \hat{\gamma}_{[t_2-\varepsilon/2,t_{3}]}\ne \emptyset$. Let $\mathcal{C}_3$ denote the remaining translates for which $\partial \hat{C}\cap Q\cap \hat{\gamma}_{[t_3,t_{1}+\varepsilon/2]}\ne \emptyset$.

We do not need to worry about the translates in $\mathcal{C}_0$, as $Q$ itself will not be monochromatic.

Take a translate $\hat C$ from $\mathcal{C}_1$ and suppose that it contains a point $p\in P_1$. By Theorem \ref{thm:multi_essw_new}, there is an edge of $\prec_{K_{1,2}}$ from a vertex of $S_{1,1}$ to $p$ and another edge from a vertex of $S_{1,2}$ to $p$. I.e., the  cone $p+K_{1,2}$  contains a point from $S_{1,1}$ and another point from $S_{1,2}$, and hence it is not monochromatic. From property (A) we know that every point in $(p+K_{1,2})\cap P$ is also in $\hat C$. Therefore, $\hat C$ is not monochromatic. 

Now consider the translates in $\mathcal{C}_2$. From property (B) we know that for these translates we have $\partial \hat C\cap Q\subset \hat{\gamma}_{[t_2-\varepsilon,t_3+\varepsilon]}$. By the definition of $t_1,t_2$ and $t_3$, we know that this implies that any two translates from $\mathcal{C}_2$ intersect at most once on their boundary within $Q$, i.e., they behave as pseudohalfplanes. To turn the translates in $\mathcal{C}_2$ into a pseudohalfplane arrangement as defined earlier, we can do as follows. For a translate $\hat{C}$, replace it with the convex set whose boundary is $\hat{\gamma}_{[t_2-\varepsilon,t_3+\varepsilon]}$  extended from its endpoints with two rays orthogonal to the segment $\overline{\hat{\gamma}(t_2-\varepsilon)\hat{\gamma}(t_3+\varepsilon)}$. This new family provides the same intersection pattern in $Q$ and forms a pseudohalfplane arrangement. We can do the same with the translates in $\mathcal{C}_3$.  Therefore, by Corollary \ref{cor:pseudohalfplane} there is a proper three-coloring for the translates in $\mathcal{C}_2\cup \mathcal{C}_3$. 
\end{proof}


 


\section{Proof of Theorem \ref{thm:multi_essw_new}}\label{sec:esswproof}

Let us quickly recap the main steps from the proof of Theorem \ref{thm:multi_essw_old} from \cite{esswproof}.
For a function $w : V \rightarrow \mathbb{R}$ and $S\subset V$ let $w(S) = \sum\limits_{x\in S}f(x)$. We say that $w: V \rightarrow [0, 1]$ is a \emph{probability distribution} on $V$ if $w(V) = 1$. 
First they carefully define a partition of the vertex set. Then for each part $P$ of the partition a probability distribution $w_P$ is defined which is concentrated on $P$. Each part of the partition is dominated independently of the other parts using a probabilistic argument. Namely, they show that we can pick some points according to $w_P$ and these points will dominate $P$ with positive probability. Also, each part is dominated using edges of just one $\prec_i$. The main reason that the proof does not immediately work for Theorem \ref{thm:multi_essw_new} is that the dominating sets for the different parts might intersect.  

Our proof of Theorem \ref{thm:multi_essw_new} follows a very similar path. We will also define a partition of the vertex set (see Figure \ref{fig:erdos}) and corresponding probability distributions.  Then we will apply the probabilistic argument for the parts simultaneously to ensure that the dominating sets of the parts are disjoint. We will apply the probabilistic argument $l$ times for each part. To be able to ensure disjointness of the $S_i^j$-s, the partition and the distributions are created a bit more carefully; we ensure that a vertex cannot have too much weight in any of the probability distributions.
This way the probability of picking any vertex in two different $S_i^j$-s will be sufficiently small.

We start with a number of useful lemmas.
The following variant of LP duality was stated in this context in \cite[Lemma 5]{MR2220666}.

\begin{lemma}[Alon et al.~\cite{MR2220666}]\label{lemma:at_least_half}
If $D=(V,A)$ is a complete multidigraph, then there exists a probability distribution $w$ on $V$ such that $w(N^-(x)) \ge 1/2$ for each $x\in V$.
\end{lemma}

We prove the following modification of Lemma \ref{lemma:at_least_half} to obtain a $w$ whose min-entropy is large. 

\begin{lemma}\label{lemma:prob_dist}
 Let $D=(V,A)$ be a complete multidigraph and let $0<\delta<1$ be a fixed number. If $|V|>\frac{1}{\delta}$, then there exists a probability distribution $w$ on $V$ such that for each $x \in V$ one of the following holds:
 \begin{itemize}
     \item $w(x)\le 2 \delta$ and $w(N^-(x)) \ge 1/2$ 
     \item  $\delta \le w(x)\le 4\delta$
 \end{itemize}
\end{lemma}
  Note that the second condition holds for at most $1/\delta$ vertices.

\begin{proof}
We define a sequence of probability distributions $w_1,\dots,w_{\floor{1/\delta}}$ and a sequence of subsets of $V$ called $R_1,\dots,R_{\floor{1/\delta}}$. Let $w_1$ be a probability distribution given by Lemma \ref{lemma:at_least_half}, that is  $w_1(N^-(x)) \ge 1/2$ for all $x\in V$.  Let $R_1=\emptyset$ if $w_1(x)<1$ for every $x\in V$ and let $R_1=\{x\}$ if $w_1$ is concentrated on a single vertex $x\in V$. 

 For $i>1$ we obtain $w_i$ by applying Lemma \ref{lemma:at_least_half} for the induced complete multidigraph $D[V\setminus R_{i-1}]$. That is, $w_i$ is a probability distribution such that $w_i(N^-(x)) \ge 1/2$ for all $x\in V\setminus R_{i-1}$ and $w_i(x)=0$ for $x\in R_{i-1}$. Let $R_i=\{x\in V | \sum\limits_{j=1}^{i} w_j(x)\ge 1\}$.
 
 To use Lemma \ref{lemma:at_least_half}, we need to check that $V\setminus R_{i-1}$ is nonempty. If $V\setminus R_{i-1}$ is empty, then by the definition of $R_{i-1}$ we have $\sum\limits_{j=1}^{i-1}w_j(x)\ge 1$ for each $x\in V$. Hence $i-1=\sum\limits_{j=1}^{i-1}w_j(V)=\sum\limits_{x\in V}\sum\limits_{j=1}^{i-1}w_j(x)\ge |V|\cdot 1$.  Since $i-1< \frac{1}{\delta}$, this contradicts the $|V|>\frac{1}{\delta}$ assumption of the lemma.

 Let $w=\frac{1}{\floor{1/\delta}}\sum\limits_{i=1}^{\floor{1/\delta}}w_i$. Clearly, this is a probability distribution on $V$.  We have to check the correctness of $w$. The first condition holds for vertices not in $R_{\floor{1/\delta}}$ and the second one for the vertices of $R_{\floor{1/\delta}}$.
 
 If $x\in V\setminus R_{\floor{1/\delta}}$, then \begin{equation*}w(N^-(x))=\frac{1}{\floor{1/\delta}}\sum\limits_{i=1}^{\floor{1/\delta}}w_i(N^-(x))\ge \frac{1}{\floor{1/\delta}} \floor{1/\delta}\frac{1}{2}=\frac{1}{2}.\end{equation*} 
 
Since $x\notin R_{\floor{1/\delta}}$ and $0<\delta<1$, we have $w(x)\le \frac{1}{\floor{1/\delta}}\cdot 1 \le 2\delta$. 
 
 In the other case, when $x\in R_{\floor{1/\delta}}$, then by definition $w(x)\ge\frac{1}{\floor{1/\delta}}\sum\limits_{i=1}^{\floor{1/\delta}}w_i(x)\ge\frac{1}{\floor{1/\delta}}\ge \delta$. On the other hand, if $j$ is the smallest number for which $x\in R_j$, then $\sum\limits_{i=1}^{j-1}w_i(x)\le 1$ and $w_k(x)=0$ for $k>j$. Therefore,
 \begin{equation*}w(x)= \frac{1}{\floor{1/\delta}}\sum\limits_{i=1}^{\floor{1/\delta}}w_i(x)= \frac{1}{\floor{1/\delta}}\sum\limits_{i=1}^{j-1}w_i(x)+\frac{1}{\floor{1/\delta}} w_j(x)\le 2\delta+2\delta=4\delta.\end{equation*}
\end{proof}

An easy consequence of Lemma \ref{lemma:prob_dist} is the following. 

\begin{lemma}\label{lemma:good}
 Let $0<\delta<1$ be fixed  and let  $D=(V,A)$ be a complete multidigraph on at least $1/\delta$ vertices whose arc set is the union of $k$ quasi orders. Then there exists a probability distribution $w$ on $V$ and a partition of $V$ into sets $T_1, T_2,\dots , T_k, R$ such that for every $i\in[k]$ and $x \in T_i$, we have $w(x)\le 2\delta$ and $w(N^-_i(x)) \ge 1/2k$ and for every $x\in R$ we have $\delta \le w(x)\le 4\delta$.
\end{lemma}

\begin{proof}
Take $w$ according to Lemma \ref{lemma:prob_dist}. For every $i\in [k]$, let $T_i'=\{x\in V|w(N^-_i(x))\ge 1/2k\}$ and let $R$ be the rest of the vertices. As $\sum_{i=1}^k w(N^-_i(x))\ge w(N^-(x)$, the sets $T_i'$ cover those vertices that satisfy the first property in Lemma \ref{lemma:prob_dist}. By making the $T_i$-s disjoint, we can obtain a partition with the required properties.
\end{proof}

\begin{proof}[Proof of Theorem \ref{thm:multi_essw_new}]
Fix $0< \delta<1$, we will choose its value later. Let $\mathcal{I}$ denote the set of sequences of length at most $k+1$ whose terms are from $[k]$, including the empty set. In other words $\mathcal{I}=\{\emptyset\}\cup\bigcup\limits_{l=1}^{k+1} [k]^l$. 

For a complete multidigraph $D=(V,A)$  we are going to define two systems of subsets $\{T_{j_1,\dots, j_i}\}_{(j_1,\dots, j_i)\in Ind_1}$ and $\{R_{j_1,\dots, j_i}\}_{(j_1,\dots, j_i)\in Ind_2}$ of the vertex set  for some index sets $Ind_1,Ind_2\subset \mathcal{I}$. The process that defines these subsets is similar to a tree traversal algorithm. The root node corresponds to $V$ and the children of a given node in the tree will correspond to a partition of the node. Therefore in the end the leaves of the tree provide a partition of the vertex set.  For each $T_{j_1,\dots, j_i}$ we also define a probability distribution $w_{j_1,\dots, j_i}$ which is concentrated on $T_{j_1,\dots, j_i}$.

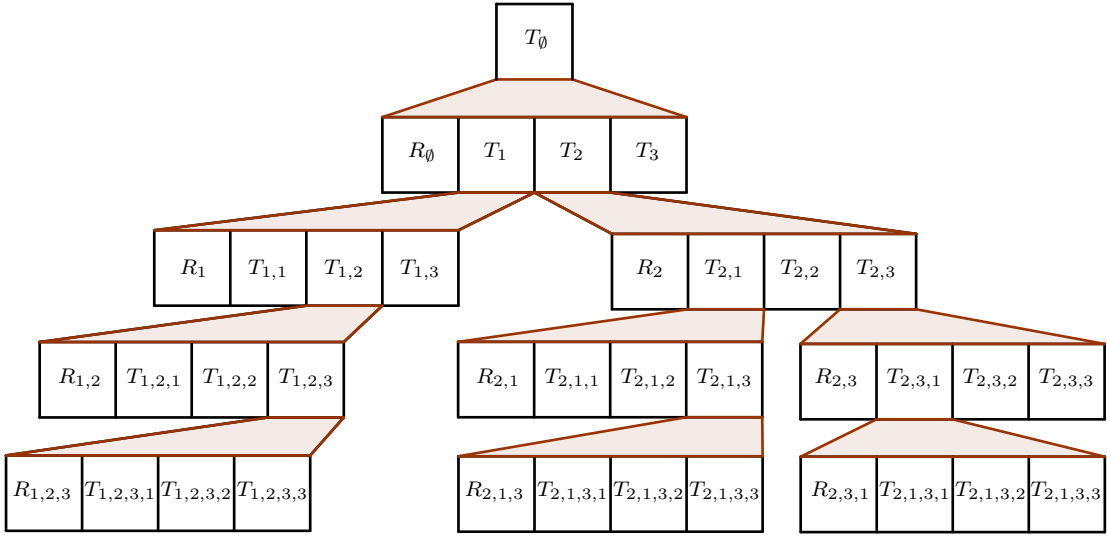
\begin{figure}[!ht]
    \centering
    \definecolor{zzttqq}{rgb}{0.6,0.2,0.}   
    \scalebox{.8}{
    \begin{tikzpicture}[line cap=round,line join=round,>=triangle 45,x=1.0cm,y=1.0cm]
    \clip(3.8304248678496586,4.360255015053259) rectangle (18.591604292782666,11.912986761629789);
    \fill[line width=\lw,color=zzttqq,fill=zzttqq,fill opacity=0.10000000149011612] (6.,8.5) -- (10.,8.5) -- (11.,9.) -- (10.,9.) -- cycle;
    \fill[line width=\lw,color=zzttqq,fill=zzttqq,fill opacity=0.10000000149011612] (12.01965123557056,8.451978418122902) -- (11.,9.) -- (12.,9.) -- (16.01965123557056,8.451978418122902) -- cycle;
    \fill[line width=\lw,color=zzttqq,fill=zzttqq,fill opacity=0.10000000149011612] (7.492555243331861,6.017748862438395) -- (4.050752571492047,5.514073585046604) -- (8.050752571492044,5.514073585046604) -- (8.49255524333186,6.017748862438395) -- cycle;
    \fill[line width=\lw,color=zzttqq,fill=zzttqq,fill opacity=0.10000000149011612] (4.49255524333186,7.017748862438395) -- (8.,7.5) -- (9.,7.5) -- (8.49255524333186,7.017748862438395) -- cycle;
    \fill[line width=\lw,color=zzttqq,fill=zzttqq,fill opacity=0.10000000149011612] (10.5,10.5) -- (9.,10.) -- (13.,10.) -- (11.5,10.5) -- cycle;
    \fill[line width=\lw,color=zzttqq,fill=zzttqq,fill opacity=0.10000000149011612] (13.01965123557056,7.4519784181229) -- (9.996168208507939,7.020411389185139) -- (13.996168208507939,7.020411389185139) -- (14.01965123557056,7.4519784181229) -- cycle;
    \fill[line width=\lw,color=zzttqq,fill=zzttqq,fill opacity=0.10000000149011612] (15.01965123557056,7.4519784181229) -- (14.498557327875192,6.992919843176194) -- (18.498557327875194,6.992919843176194) -- (16.01965123557056,7.4519784181229) -- cycle;
    \fill[line width=\lw,color=zzttqq,fill=zzttqq,fill opacity=0.10000000149011612] (12.996168208507939,6.020411389185139) -- (10.000773758264039,5.502749215070639) -- (14.000773758264039,5.502749215070639) -- (13.996168208507939,6.020411389185139) -- cycle;
    \fill[line width=\lw,color=zzttqq,fill=zzttqq,fill opacity=0.10000000149011612] (15.498557327875192,5.992919843176194) -- (14.501495044177325,5.497236961707589) -- (18.501495044177325,5.497236961707589) -- (16.498557327875194,5.992919843176194) -- cycle;
    \draw [line width=\lw] (9.,10.)-- (9.,9.);
    \draw [line width=\lw] (9.,9.)-- (13.,9.);
    \draw [line width=\lw] (13.,9.)-- (13.,10.);
    \draw [line width=\lw] (13.,10.)-- (9.,10.);
    \draw [line width=\lw] (10.,10.)-- (10.,9.);
    \draw [line width=\lw] (11.,10.)-- (11.,9.);
    \draw [line width=\lw] (12.,10.)-- (12.,9.);
    \draw [line width=\lw] (10.,9.)-- (6.,8.5);
    \draw [line width=\lw] (11.,9.)-- (10.,8.5);
    \draw [line width=\lw] (6.,8.5)-- (10.,8.5);
    \draw [line width=\lw] (10.,8.5)-- (10.,7.5);
    \draw [line width=\lw] (10.,7.5)-- (6.,7.5);
    \draw [line width=\lw] (6.,7.5)-- (6.,8.5);
    \draw [line width=\lw] (7.,8.5)-- (7.,7.5);
    \draw [line width=\lw] (8.,8.5)-- (8.,7.5);
    \draw [line width=\lw] (9.,8.5)-- (9.,7.5);
    \draw [line width=\lw] (11.,9.)-- (12.01965123557056,8.451978418122902);
    \draw [line width=\lw] (12.01965123557056,8.451978418122902)-- (12.01965123557056,7.4519784181229);
    \draw [line width=\lw] (12.,9.)-- (16.01965123557056,8.451978418122902);
    \draw [line width=\lw] (16.01965123557056,8.451978418122902)-- (16.01965123557056,7.4519784181229);
    \draw [line width=\lw] (16.01965123557056,7.4519784181229)-- (12.01965123557056,7.4519784181229);
    \draw [line width=\lw] (12.01965123557056,8.451978418122902)-- (16.01965123557056,8.451978418122902);
    \draw [line width=\lw] (13.01965123557056,8.451978418122902)-- (13.01965123557056,7.4519784181229);
    \draw [line width=\lw] (14.01965123557056,8.451978418122902)-- (14.01965123557056,7.4519784181229);
    \draw [line width=\lw] (15.01965123557056,8.451978418122902)-- (15.01965123557056,7.4519784181229);
    \draw (9.2,9.8) node[anchor=north west] {\scriptsize $R_\emptyset$};
    \draw (11.2,9.8) node[anchor=north west] {\scriptsize $T_2$};
    \draw (7.1,8.25) node[anchor=north west] {\scriptsize $T_{1,1}$};
    \draw (8.1,8.25) node[anchor=north west] {\scriptsize $T_{1,2}$};
    \draw (12.2,8.25) node[anchor=north west] {\scriptsize $R_2$};
    \draw (15.1,8.25) node[anchor=north west] {\scriptsize $T_{2,3}$};
    \draw (14.1,8.25) node[anchor=north west] {\scriptsize $T_{2,2}$};
    \draw (13.1,8.25) node[anchor=north west] {\scriptsize $T_{2,1}$};
    \draw (9.1,8.25) node[anchor=north west] {\scriptsize $T_{1,3}$};
    \draw (12.2,9.8) node[anchor=north west] {\scriptsize $T_3$};
    \draw (6.2,8.25) node[anchor=north west] {\scriptsize $R_1$};
    \draw (10.2,9.8) node[anchor=north west] {\scriptsize $T_1$};
    \draw [line width=\lw,color=zzttqq] (6.,8.5)-- (10.,8.5);
    \draw [line width=\lw,color=zzttqq] (10.,8.5)-- (11.,9.);
    \draw [line width=\lw,color=zzttqq] (11.,9.)-- (10.,9.);
    \draw [line width=\lw,color=zzttqq] (10.,9.)-- (6.,8.5);
    \draw [line width=\lw,color=zzttqq] (12.01965123557056,8.451978418122902)-- (11.,9.);
    \draw [line width=\lw,color=zzttqq] (11.,9.)-- (12.,9.);
    \draw [line width=\lw,color=zzttqq] (12.,9.)-- (16.01965123557056,8.451978418122902);
    \draw [line width=\lw,color=zzttqq] (16.01965123557056,8.451978418122902)-- (12.01965123557056,8.451978418122902);
    \draw [line width=\lw] (4.050752571492047,4.514073585046604)-- (4.050752571492047,5.514073585046604);
    \draw [line width=\lw] (4.050752571492047,5.514073585046604)-- (8.050752571492044,5.514073585046604);
    \draw [line width=\lw] (8.050752571492044,5.514073585046604)-- (8.050752571492044,4.514073585046604);
    \draw [line width=\lw] (4.050752571492047,4.514073585046604)-- (8.050752571492044,4.514073585046604);
    \draw [line width=\lw] (5.050752571492047,5.514073585046604)-- (5.050752571492047,4.514073585046604);
    \draw [line width=\lw] (6.050752571492044,5.514073585046604)-- (6.050752571492044,4.514073585046604);
    \draw [line width=\lw] (7.050752571492042,5.514073585046604)-- (7.050752571492042,4.514073585046604);
    \draw (4.0,5.3) node[anchor=north west] {\scriptsize $R_{1,2,3}$};
    \draw (4.95,5.3) node[anchor=north west] {\scriptsize $T_{1,2,3,1}$};
    \draw (5.95,5.3) node[anchor=north west] {\scriptsize $T_{1,2,3,2}$};
    \draw (6.95,5.3) node[anchor=north west] {\scriptsize $T_{1,2,3,3}$};
    \draw [line width=\lw] (4.49255524333186,6.017748862438395)-- (4.49255524333186,7.017748862438395);
    \draw [line width=\lw] (4.49255524333186,7.017748862438395)-- (8.49255524333186,7.017748862438395);
    \draw [line width=\lw] (8.49255524333186,7.017748862438395)-- (8.49255524333186,6.017748862438395);
    \draw [line width=\lw] (4.49255524333186,6.017748862438395)-- (8.49255524333186,6.017748862438395);
    \draw [line width=\lw] (5.49255524333186,7.017748862438395)-- (5.49255524333186,6.017748862438395);
    \draw [line width=\lw] (6.492555243331861,7.017748862438395)-- (6.492555243331861,6.017748862438395);
    \draw [line width=\lw] (7.492555243331861,7.017748862438395)-- (7.492555243331861,6.017748862438395);
    \draw (4.6,6.8) node[anchor=north west] {\scriptsize $R_{1,2}$};
    \draw (5.5,6.8) node[anchor=north west] {\scriptsize $T_{1,2,1}$};
    \draw (6.5,6.8) node[anchor=north west] {\scriptsize $T_{1,2,2}$};
    \draw (7.5,6.8) node[anchor=north west] {\scriptsize $T_{1,2,3}$};
    \draw [line width=\lw] (4.49255524333186,7.017748862438395)-- (8.,7.5);
    \draw [line width=\lw] (9.,7.5)-- (8.49255524333186,7.017748862438395);
    \draw [line width=\lw] (7.492555243331861,6.017748862438395)-- (4.050752571492047,5.514073585046604);
    \draw [line width=\lw] (8.49255524333186,6.017748862438395)-- (8.050752571492044,5.514073585046604);
    \draw [line width=\lw,color=zzttqq] (7.492555243331861,6.017748862438395)-- (4.050752571492047,5.514073585046604);
    \draw [line width=\lw,color=zzttqq] (4.050752571492047,5.514073585046604)-- (8.050752571492044,5.514073585046604);
    \draw [line width=\lw,color=zzttqq] (8.050752571492044,5.514073585046604)-- (8.49255524333186,6.017748862438395);
    \draw [line width=\lw,color=zzttqq] (8.49255524333186,6.017748862438395)-- (7.492555243331861,6.017748862438395);
    \draw [line width=\lw,color=zzttqq] (4.49255524333186,7.017748862438395)-- (8.,7.5);
    \draw [line width=\lw,color=zzttqq] (8.,7.5)-- (9.,7.5);
    \draw [line width=\lw,color=zzttqq] (9.,7.5)-- (8.49255524333186,7.017748862438395);
    \draw [line width=\lw,color=zzttqq] (8.49255524333186,7.017748862438395)-- (4.49255524333186,7.017748862438395);
    \draw [line width=\lw,color=zzttqq] (10.5,10.5)-- (9.,10.);
    \draw [line width=\lw,color=zzttqq] (9.,10.)-- (13.,10.);
    \draw [line width=\lw,color=zzttqq] (13.,10.)-- (11.5,10.5);
    \draw [line width=\lw,color=zzttqq] (11.5,10.5)-- (10.5,10.5);
    \draw [line width=\lw] (10.5,10.5)-- (10.5,11.5);
    \draw [line width=\lw] (10.5,11.5)-- (11.5,11.5);
    \draw [line width=\lw] (11.5,11.5)-- (11.5,10.5);
    \draw (10.737623986513645,11.3) node[anchor=north west] {\scriptsize  $T_{\emptyset}$};
    \draw [line width=\lw] (14.498557327875192,5.992919843176194)-- (14.498557327875192,6.992919843176194);
    \draw [line width=\lw] (14.498557327875192,6.992919843176194)-- (18.498557327875194,6.992919843176194);
    \draw [line width=\lw] (18.498557327875194,6.992919843176194)-- (18.498557327875194,5.992919843176194);
    \draw [line width=\lw] (14.498557327875192,5.992919843176194)-- (18.498557327875194,5.992919843176194);
    \draw [line width=\lw] (15.498557327875192,6.992919843176194)-- (15.498557327875192,5.992919843176194);
    \draw [line width=\lw] (16.498557327875194,6.992919843176194)-- (16.498557327875194,5.992919843176194);
    \draw [line width=\lw] (17.498557327875194,6.992919843176194)-- (17.498557327875194,5.992919843176194);
    \draw (14.55,6.8) node[anchor=north west] {\scriptsize $R_{2,3}$};
    \draw (15.5,6.8) node[anchor=north west] {\scriptsize $T_{2,3,1}$};
    \draw (16.5,6.8) node[anchor=north west] {\scriptsize $T_{2,3,2}$};
    \draw (17.5,6.8) node[anchor=north west] {\scriptsize $T_{2,3,3}$};
    \draw [line width=\lw] (9.996168208507939,6.020411389185139)-- (9.996168208507939,7.020411389185139);
    \draw [line width=\lw] (9.996168208507939,7.020411389185139)-- (13.996168208507939,7.020411389185139);
    \draw [line width=\lw] (13.996168208507939,7.020411389185139)-- (13.996168208507939,6.020411389185139);
    \draw [line width=\lw] (9.996168208507939,6.020411389185139)-- (13.996168208507939,6.020411389185139);
    \draw [line width=\lw] (10.996168208507939,7.020411389185139)-- (10.996168208507939,6.020411389185139);
    \draw [line width=\lw] (11.996168208507939,7.020411389185139)-- (11.996168208507939,6.020411389185139);
    \draw [line width=\lw] (12.996168208507939,7.020411389185139)-- (12.996168208507939,6.020411389185139);
    \draw (10.1,6.8) node[anchor=north west] {\scriptsize $R_{2,1}$};
    \draw (11.0,6.8) node[anchor=north west] {\scriptsize $T_{2,1,1}$};
    \draw (12.0,6.8) node[anchor=north west] {\scriptsize $T_{2,1,2}$};
    \draw (13.0,6.8) node[anchor=north west] {\scriptsize $T_{2,1,3}$};
    \draw [line width=\lw,color=zzttqq] (13.01965123557056,7.4519784181229)-- (9.996168208507939,7.020411389185139);
    \draw [line width=\lw,color=zzttqq] (9.996168208507939,7.020411389185139)-- (13.996168208507939,7.020411389185139);
    \draw [line width=\lw,color=zzttqq] (13.996168208507939,7.020411389185139)-- (14.01965123557056,7.4519784181229);
    \draw [line width=\lw,color=zzttqq] (14.01965123557056,7.4519784181229)-- (13.01965123557056,7.4519784181229);
    \draw [line width=\lw,color=zzttqq] (15.01965123557056,7.4519784181229)-- (14.498557327875192,6.992919843176194);
    \draw [line width=\lw,color=zzttqq] (14.498557327875192,6.992919843176194)-- (18.498557327875194,6.992919843176194);
    \draw [line width=\lw,color=zzttqq] (18.498557327875194,6.992919843176194)-- (16.01965123557056,7.4519784181229);
    \draw [line width=\lw,color=zzttqq] (16.01965123557056,7.4519784181229)-- (15.01965123557056,7.4519784181229);
    \draw [line width=\lw] (14.501495044177325,4.497236961707589)-- (14.501495044177325,5.497236961707589);
    \draw [line width=\lw] (14.501495044177325,5.497236961707589)-- (18.501495044177325,5.497236961707589);
    \draw [line width=\lw] (18.501495044177325,5.497236961707589)-- (18.501495044177325,4.497236961707589);
    \draw [line width=\lw] (14.501495044177325,4.497236961707589)-- (18.501495044177325,4.497236961707589);
    \draw [line width=\lw] (15.501495044177325,5.497236961707589)-- (15.501495044177325,4.497236961707589);
    \draw [line width=\lw] (16.501495044177325,5.497236961707589)-- (16.501495044177325,4.497236961707589);
    \draw [line width=\lw] (17.501495044177325,5.497236961707589)-- (17.501495044177325,4.497236961707589);
    \draw (14.5,5.3) node[anchor=north west] {\scriptsize $R_{2,3,1}$};
    \draw (15.4,5.3) node[anchor=north west] {\scriptsize $T_{2,1,3,1}$};
    \draw (16.4,5.3) node[anchor=north west] {\scriptsize $T_{2,1,3,2}$};
    \draw (17.39,5.3) node[anchor=north west] {\scriptsize $T_{2,1,3,3}$};
    \draw [line width=\lw] (10.000773758264039,4.502749215070639)-- (10.000773758264039,5.502749215070639);
    \draw [line width=\lw] (10.000773758264039,5.502749215070639)-- (14.000773758264039,5.502749215070639);
    \draw [line width=\lw] (14.000773758264039,5.502749215070639)-- (14.000773758264039,4.502749215070639);
    \draw [line width=\lw] (10.000773758264039,4.502749215070639)-- (14.000773758264039,4.502749215070639);
    \draw [line width=\lw] (11.000773758264039,5.502749215070639)-- (11.000773758264039,4.502749215070639);
    \draw [line width=\lw] (12.000773758264039,5.502749215070639)-- (12.000773758264039,4.502749215070639);
    \draw [line width=\lw] (13.000773758264039,5.502749215070639)-- (13.000773758264039,4.502749215070639);
    \draw (9.95,5.3) node[anchor=north west] {\scriptsize $R_{2,1,3}$};
    \draw (10.9,5.3) node[anchor=north west] {\scriptsize $T_{2,1,3,1}$};
    \draw (11.9,5.3) node[anchor=north west] {\scriptsize $T_{2,1,3,2}$};
    \draw (12.9,5.3) node[anchor=north west] {\scriptsize $T_{2,1,3,3}$};
    \draw [line width=\lw,color=zzttqq] (12.996168208507939,6.020411389185139)-- (10.000773758264039,5.502749215070639);
    \draw [line width=\lw,color=zzttqq] (10.000773758264039,5.502749215070639)-- (14.000773758264039,5.502749215070639);
    \draw [line width=\lw,color=zzttqq] (14.000773758264039,5.502749215070639)-- (13.996168208507939,6.020411389185139);
    \draw [line width=\lw,color=zzttqq] (13.996168208507939,6.020411389185139)-- (12.996168208507939,6.020411389185139);
    \draw [line width=\lw,color=zzttqq] (15.498557327875192,5.992919843176194)-- (14.501495044177325,5.497236961707589);
    \draw [line width=\lw,color=zzttqq] (14.501495044177325,5.497236961707589)-- (18.501495044177325,5.497236961707589);
    \draw [line width=\lw,color=zzttqq] (18.501495044177325,5.497236961707589)-- (16.498557327875194,5.992919843176194);
    \draw [line width=\lw,color=zzttqq] (16.498557327875194,5.992919843176194)-- (15.498557327875192,5.992919843176194);
    \end{tikzpicture}}
    \caption{The outcome of the partition process for $k=3$ assuming that the size of $T_3$ and $T_{1,3}$ is less than $1/\delta$.}
    \label{fig:erdos}
\end{figure}

We start by setting $T_\emptyset=V$ and we apply Lemma \ref{lemma:good} for $T_\emptyset$ and our fixed $\delta$ to obtain $T_1, T_2, \dots , T_k$ and $R_\emptyset$ together with a probability distribution $w_\emptyset$. 

Then, as long as possible, we do the following step. We pick a previously not selected $T_{j_1,\dots, j_i}$ from the already defined ones such that $j_1,\dots,j_{i}$ are pairwise distinct, and $|T_{j_1,\dots, j_i}|>\frac{1}{\delta}$. For such a $T_{j_1,\dots, j_i}$ we apply Lemma $\ref{lemma:good}$ to obtain the partition $T_{j_1,\dots, j_i,1},\dots,T_{j_1,\dots, j_i,k},R_{j_1,\dots, j_i}$ and a probability distribution $w_{j_1,\dots, j_i}$. This process terminates in (strictly) less than $|\mathcal{I}|\le k^{k+2}$ steps as no index sequence can be longer than $k+1$. Let $Ind_1$ and $Ind_2$ denote the index sets for the $T_{j_1,\dots, j_i}$-s and $R_{j_1,\dots, j_i}$-s. From Lemma \ref{lemma:good} we know the following:

\begin{itemize}
    \item $|R_{j_1,\dots, j_i}|\le \frac{1}{\delta}$ for each $(j_1,\dots,j_i)\in Ind_2$.
    \item For every $r\in [k]$ and $x \in T_{j_1,\dots, j_i,r}$ we have $w_{j_1,\dots, j_i}(N^-_r(x)) \ge 1/2k$.
    \item $w_{j_1,\dots, j_i}(x)\le 4\delta$ for every $x\in V$ and every $(j_1,\dots,j_i)\in Ind_2$.
\end{itemize}

When the process halts, we define a partition of $V$ by taking every set corresponding to a leaf, i.e., every set of the following three kinds:
\begin{enumerate}
    \item $R_{j_1,\dots,j_i}$-s
    \item  $T_{j_1,\dots,j_i}$-s that have fewer elements than $\frac{1}{\delta}$
    \item $T_{j_1,\dots,j_i}$-s where $j_i=j_r$ for some $r<i$
\end{enumerate}
    
 These sets are clearly disjoint, and they cover the vertex set, since the process halted. Let $R$ denote the union of the $R_{j_1,\dots,j_i}$-s, let  $T_{small}$ denote the union the $T_{j_1,\dots,j_i}$-s that have fewer elements than $1/\delta$ and finally let $T_{rep}$ denote the union of the $T_{j_1,\dots,j_i}$-s where $j_i=j_r$ for some $r<i$.

We dominate each part of the partition. Since each $R_{j_1,\dots,j_i}$ has fewer than $\frac{1}{\delta}$ elements, we have $|R|\le  \frac{1}{\delta}|Ind_2|\le \frac{1}{\delta}k^{k+2}$. Similarly $|T_{small}|\le \frac{1}{\delta}k^{k+2}$. We will simply put the vertices of $R\cup T_{small}$ into $S_1^1$ but we will not use them to dominate any point. Later, when we pick $\delta$, we will see that $\frac{2}{\delta}k^{k+2}$ is upper bounded by a function of $k$ and $l$, so we have not added too many vertices here. 

It remains to dominate the vertices in $T_{rep}$. Let $I_{rep}\subset Ind_1$ be the index set of those $T_{j_1,\dots,j_i}$-s where $j_i=j_r$ for some $r<i$. 

Let us fix $(j_1,\dots,j_i)\in I_{rep}$ and $o\in[l]$. First we explain how to chose a set $V_{j_1,\dots,j_i}^o\subset V$ such that $V_{j_1,\dots,j_i}^o$ dominates $T_{j_1,\dots,j_i}$ and then we will argue that this can be done simultaneously for every $(j_1,\dots,j_i)\in I_{rep}$ and $o\in[l]$ in such a way that the $V_{j_1,\dots,j_i}^o$-s are disjoint. 

For a fixed $(j_1,\dots,j_i)\in I_{rep}$, consider $T_{j_1,\dots,j_i}\subset T_{j_1,\dots,j_{i-1}}\subset T_{j_1,\dots,j_{r-1}}$ such that $j_i=j_r$. The idea is that if we appropriately pick $V_{j_1,\dots,j_i}^o\subset T_{j_1,\dots,j_{r-1}}$, such that $V_{j_1,\dots,j_i}^o$ dominates a large fraction of $T_{j_1,\dots,j_{i-1}}$ through the edges of $\prec_{j_r}$, then that large fraction dominates entirely $T_{j_1,\dots,j_{i}}$, and since $j_r=j_i$, this happens through the edges of $\prec_{j_i}$ again. By transitivity, this implies that $T_{j_1,\dots,j_{i}}$ is dominated by $V_{j_1,\dots,j_i}^o$. To put this idea into precise terms note that for each $x\in T_{j_1,\dots,j_{i}}$ we have $w_{j_1,\dots,j_{i-1}}(N_{j_i}^-(x)) \ge \frac{1}{2k}$. Hence, if $w_{j_1,\dots,j_{i-1}}(N^+_{j_i}(V_{j_1,\dots,j_{i}}^o))> 1-\frac{1}{2k}$, then $T_{j_1,\dots,j_{i}}$ is dominated by $V_{j_1,\dots,j_{i}}^o$.  

Let $0<\varepsilon<1$ be fixed, we will choose its value later. Let $g(\varepsilon)=\floor{\frac{\ln(\varepsilon)}{\ln(1-\frac{1}{2k})}}+1$  and let  $V_{j_1,\dots,j_i}^o$ be a multiset of $g(\varepsilon)$ elements picked independently at random from $T_{j_1,\dots,j_{r-1}}$ according to the distribution $w_{j_1,\dots,j_{r-1}}$. For every vertex $x \in T_{j_1,\dots,j_{i-1}}$, $\mathbb{P}(x\in N_{j_i}^+(V_{j_1,\dots,j_{i}}^o)) \ge 1- (1- \frac{1}{2k})^{g(\varepsilon)}\ge 1- \varepsilon$. Therefore, by linearity of expectation,
\begin{equation*}\mathbb{E}(w_{j_1,\dots,j_{i-1}}(N^+_{j_i}(V_{j_1,\dots,j_{i}}^o)))\ge\sum_{x\in T_{j_1,\dots,j_{i-1}}}w_{j_1,\dots,j_{i-1}}(x)\cdot(1-\varepsilon)\ge 1-\varepsilon.\end{equation*}

This already shows that if $\varepsilon\le 1/2k$, there is a positive probability that $V_{j_1,\dots,j_{i}}^o$ dominates $T_{j_1,\dots,j_{i}}$. But we will need a smaller $\varepsilon$ for disjointness. 

 Let $X$ be the random variable $\sum\limits_{(j_1,\dots,j_{i})\in I_{rep},o\in [l]}w_{j_1,\dots,j_{i-1}}(N^+_{j_i}(V_{j_1,\dots,j_{i}}^o)) $.
 By our reasoning above,
$\mathbb{E}(X)\ge l|I_{rep}|(1-\varepsilon).$

Let $Y$ be the indicator variable of the event that the sets $V_{j_1,\dots,j_{i}}^o$ for $(j_1,\dots,j_{i})\in I_{rep},o\in [l]$ are pairwise disjoint. Let us consider the random variable $X\cdot Y$. Suppose that $\mathbb{P}(X\cdot Y>l|I_{rep}|-\frac{1}{2k})>0$. This would imply that $Y=1$, i.e., the $V_{j_1,\dots,j_{i}}^o$-s are pairwise disjoint. Since $w_{j_1,\dots,j_{i-1}}(N^+_{j_i}(V_{j_1,\dots,j_{i}}^o))\le 1$ for each $(j_1,\dots,j_i)\in I_{rep}$, it would also imply that $w_{j_1,\dots,j_{i-1}}(N^+_{j_i}(V_{j_1,\dots,j_{i}}^o))> 1-\frac{1}{2k}$ for each $(j_1,\dots,j_i)\in I_{rep}$ and $o\in[l]$.
From here it will be easy to construct the required sets $S_i^j$.

We will show an lower bound on $\mathbb{E}(X\cdot Y)$. Since $Y$ is an indicator variable, we have
\begin{equation*}\mathbb{E}(X\cdot Y)\ge \mathbb{E}(X)-\max(X)\cdot\mathbb{P}(Y=0).\end{equation*}

Clearly $\max(X)\le l|I_{rep}|$. To bound $\mathbb{P}(Y=0)$, note that in total we have picked at most $N=l|I_{rep}|g(\epsilon)\le l k^{k+2}g(\epsilon)$ elements. Every time we picked an element, the probability distribution was smaller than $4\delta$ on each element of $V$. Hence
\begin{equation*}\mathbb{P}(Y=0)\le\binom{N}{2}4\delta\le N^2 2\delta\end{equation*} 

Therefore,
\begin{equation*}\mathbb{E}(X\cdot Y)\ge \mathbb{E}(X)-\max(X)\cdot\mathbb{P}(Y=0)\ge\end{equation*}
\begin{equation*}\ge  l |I_{rep}|(1-\varepsilon)- l |I_{rep}|N^2 2\delta  =l |I_{rep}|(1-\varepsilon-N^2 2\delta)\end{equation*}
Pick $\varepsilon=\frac{1}{4l k^{k+3}}$ and $\delta=\frac{1}{8 l^3 k^{3k+7}g(\varepsilon)^2}$. With this choice  $N^2 2\delta\le \frac{2l^2k^{2k+4}g(\varepsilon)^2}{8l^3 k^{3k+7} g(\varepsilon)^2}=\frac{1}{4l k^{k+3}}$.

Therefore, 
\begin{equation*}\mathbb{E}(X\cdot Y) \ge l |I_{rep}|(1-\varepsilon-N^2 2\delta)\ge l |I_{rep}|(1-\frac{1}{4l k^{k+3}}-\frac{1}{4l k^{k+3}})\end{equation*}
\begin{equation*}\ge  l|I_{rep}|(1-\frac{1}{2k l| I_{rep}|})=l |I_{rep}|-\frac{1}{2k}.\end{equation*}


 Hence, it is possible that $X\cdot Y\ge l |I_{rep}|-\frac{1}{2k}$. In this case the $V_{j_1,\dots,j_{i}}^o$-s are pairwise disjoint and each $T_{j_1,\dots,j_{i}}$ is dominated by $V_{j_1,\dots,j_{i}}^o$ through the edges of $\prec_{j_i}$. 
 
Finally, let $S_i^j=\bigcup\limits_{({j_1,\dots,j_r,i})\in I_{rep}} V_{j_1,\dots,j_r,i}^j$ (be careful that $i$ denotes something else than in the paragraph above). The $S_i^j$-s are also pairwise disjoint. For any vertex  $v\in V\setminus \bigcup\limits_{i,j}S_{i}^j\subset T_{rep}$ there is a $T_{j_1,\dots,j_r,i}$ containing it for some $({j_1,\dots,j_r,i})\in I_{rep}$ (recall that we added $R\cup T_{small}$ to $S_1^1$), hence there is an edge of $\prec_{i}$ from $v$ to a vertex of $S^j_{i}$ for each $j\in[l]$. 

To wrap up the proof, we calculate the size of $\bigcup\limits_{i,j}S_{i}^j$. We placed $R\cup T_{small}$ into $S_1^1$ and we have at most $g(\varepsilon)$ elements in each $V_{j_1,\dots,j_i}^o$. Therefore, we have at most \begin{equation*}\frac{2}{\delta}k^{k+2}+l k^{k+2}g(\varepsilon)\le 2 \cdot{8 l^3 k^{3k+7}(\floor{\frac{\ln(\frac{1}{4l k^{k+3}})}{\ln(1-\frac{1}{2k})}}+1)^2} k^{k+2}+l k^{k+2}(\floor{\frac{\ln(\frac{1}{4l k^{k+3}})}{\ln(1-\frac{1}{2k})}}+1)\end{equation*} elements in $\bigcup\limits_{i,j}S_{i}^j$.
\end{proof}

\section{Open questions}\label{sec:open}

It is a natural question whether there is a universal $m$ that works for all convex bodies in Theorem \ref{thm:general_three_col}, like in Theorem \ref{thm:local_three_col}.
This would follow if we could choose $r$ to be a universal constant.
While the $r$ given by our algorithm can depend on $C$, we can apply an appropriate affine transformation to $C$ before choosing $r$; this does not change the hypergraphs that can be realized with the range space determined by the translates of $C$.
To ensure that properties (A) and (B) are satisfied would require further study of the Illumination conjecture.

Our bound for $m$ is quite large, even for the unit disk, both in Theorems \ref{thm:general_three_col} and \ref{thm:local_three_col}, which is mainly due to the fact that $f(3,2)$ given by Theorem \ref{thm:multi_essw_new} is huge.
It has been conjectured that in Theorem \ref{thm:multi_essw_old} the optimal value is $f(3)=3$, and a similarly small number seems realistic for $f(3,2)$ as well.

While Theorem \ref{thm:general_three_col} closed the last main question left open for primal hypergraphs realizable by translates of planar convex bodies, the respective problem is still open in higher dimensions.
While it is not hard to show that some hypergraphs with high chromatic number often used in constructions can be easily realized by unit balls in $\mathbb{R}^5$, we do not know whether the chromatic number is bounded or not in $\mathbb{R}^3$.
From our Union Lemma (Lemma \ref{lem:combine}) it follows that to establish boundedness, it would be enough to find a polychromatic $k$-coloring for pseudohalfspaces, whatever this word means.



\begin{funding}
The work of DP was partially supported by the ERC Advanced Grant ``ERMiD'' and by the Lend\"ulet program LP2017-19/2017 and the J\'anos Bolyai Research Scholarship of the Hungarian Academy of Sciences, and by the New National Excellence Program \'UNKP-22-5 and by the Thematic Excellence Program TKP2021-NKTA-62 of the National Research, Development and Innovation Office.

The work of GD was partially supported by the ÚNKP-21-3 New National Excellence Program of the
Ministry for Innovation and Technology from the source of the National Research, Development
and Innovation Fund, and by the Lend\"ulet program LP2017-19/2017, and by ERC Advanced Grant ``GeoScape”. 
\end{funding}



\begin{thebibliography}{99}

\bibitem{MR4169259}
Eyal Ackerman, Bal\'{a}zs Keszegh, and D\"{o}m\"{o}t\"{o}r P\'{a}lv\"{o}lgyi.
\newblock Coloring hypergraphs defined by stabbed pseudo-disks and {ABAB}-free
  hypergraphs.
\newblock {\em SIAM J. Discrete Math.}, 34(4):2250--2269, 2020.
\newblock \href {https://doi.org/10.1137/19M1290231}
  {\path{doi:10.1137/19M1290231}}.

\bibitem{homotsquare}
Eyal Ackerman, Bal\'{a}zs Keszegh, and M\'at\'e Vizer.
\newblock Coloring points with respect to squares.
\newblock {\em Discrete Comput. Geom.}, 58(4):757--784, 2017.
\newblock \href {https://doi.org/10.1007/s00454-017-9902-y}
  {\path{doi:10.1007/s00454-017-9902-y}}.


\bibitem{MR2220666}Noga Alon, Graham Brightwell, Henry Kierstead, Alexandr Kostochka and Peter Winkler. Dominating sets in k-majority tournaments. {\em J. Combin. Theory Ser. B}. 96, 374-387 (2006).
\newblock \href {https://doi.org/10.1016/j.jctb.2005.09.003}
  {\path{doi:10.1016/j.jctb.2005.09.003}}.


\bibitem{Alou}
Greg Aloupis, Jean Cardinal, S\'{e}bastien Collette, Stefan Langerman, David
  Orden, and Pedro Ramos.
\newblock Decomposition of multiple coverings into more parts.
\newblock {\em Discrete Comput. Geom.}, 44(3):706--723, 2010.
\newblock \href {https://doi.org/10.1007/s00454-009-9238-3}
  {\path{doi:10.1007/s00454-009-9238-3}}.

\bibitem{MR3126347}
Andrei Asinowski, Jean Cardinal, Nathann Cohen, S\'{e}bastien Collette, Thomas
  Hackl, Michael Hoffmann, Kolja Knauer, Stefan Langerman, Micha\l Laso\'{n},
  Piotr Micek, G\"{u}nter Rote, and Torsten Ueckerdt.
\newblock Coloring hypergraphs induced by dynamic point sets and bottomless
  rectangles.
\newblock In {\em Algorithms and data structures}, volume 8037 of {\em Lecture
  Notes in Comput. Sci.}, pages 73--84. Springer, Heidelberg, 2013.
\newblock \href {https://doi.org/10.1007/978-3-642-40104-6_7}
  {\path{doi:10.1007/978-3-642-40104-6_7}}.

\bibitem{MR3816868}
K\'{a}roly Bezdek and Muhammad~A. Khan.
\newblock The geometry of homothetic covering and illumination.
\newblock In {\em Discrete geometry and symmetry}, volume 234 of {\em Springer
  Proc. Math. Stat.}, pages 1--30. Springer, Cham, 2018.
\newblock \href {https://doi.org/10.1007/978-3-319-78434-2_1}
  {\path{doi:10.1007/978-3-319-78434-2_1}}.

\bibitem{esswproof}
Nicolas Bousquet, William Lochet, and St\'{e}phan Thomass\'{e}.
\newblock A proof of the {E}rd{\H o}s-{S}ands-{S}auer-{W}oodrow conjecture.
\newblock {\em J. Combin. Theory Ser. B}, 137:316--319, 2019.
\newblock \href {https://doi.org/10.1016/j.jctb.2018.11.005}
  {\path{doi:10.1016/j.jctb.2018.11.005}}.

\bibitem{MR3151767}
Jean Cardinal, Kolja Knauer, Piotr Micek, and Torsten Ueckerdt.
\newblock Making triangles colorful.
\newblock {\em J. Comput. Geom.}, 4(1):240--246, 2013.
\newblock \href {https://doi.org/10.20382/jocg.v4i1a10}
  {\path{doi:10.20382/jocg.v4i1a10}}.

\bibitem{CKMU13}
Jean {Cardinal}, Kolja {Knauer}, Piotr {Micek}, and Torsten {Ueckerdt}.
\newblock {Making octants colorful and related covering decomposition
  problems}.
\newblock {\em {SIAM J. Discrete Math.}}, 28(4):1948--1959, 2014.
\newblock \href {https://doi.org/10.1137/140955975}
  {\path{doi:10.1137/140955975}}.

\bibitem{CKMPUV20}
Jean Cardinal, Piotr~Micek Kolja~Knauer, Dömötör Pálvölgyi, Torsten
  Ueckerdt, and Narmada Varadarajan.
\newblock Colouring bottomless rectangles and arborescences.
\newblock {\em To appear}, 2020.

\bibitem{Chen}
Xiaomin Chen, J\'{a}nos Pach, Mario Szegedy, and G\'{a}bor Tardos.
\newblock Delaunay graphs of point sets in the plane with respect to
  axis-parallel rectangles.
\newblock {\em Random Structures Algorithms}, 34(1):11--23, 2009.
\newblock \href {https://doi.org/10.1002/rsa.20246}
  {\path{doi:10.1002/rsa.20246}}.

\bibitem{threechromaticdisk}
G\'{a}bor Dam\'{a}sdi and D{\"{o}}m{\"{o}}t{\"{o}}r P{\'{a}}lv{\"{o}}lgyi.
\newblock Unit disks hypergraphs are three-colorable.
\newblock {\em Extended Abstracts EuroComb 2021, Trends in Mathematics},
  14:483--489, 2021.

\bibitem{fourchromatic}
G\'{a}bor Dam\'{a}sdi and D{\"{o}}m{\"{o}}t{\"{o}}r P{\'{a}}lv{\"{o}}lgyi.
\newblock Realizing an m-uniform four-chromatic hypergraph with disks.
\newblock {\em Combinatorica, to appear}, 2022.

\bibitem{MR2812512}
Matt Gibson and Kasturi Varadarajan.
\newblock Optimally decomposing coverings with translates of a convex polygon.
\newblock {\em Discrete Comput. Geom.}, 46(2):313--333, 2011.
\newblock \href {https://doi.org/10.1007/s00454-011-9353-9}
  {\path{doi:10.1007/s00454-011-9353-9}}.

\bibitem{wcf2}
Bal\'{a}zs Keszegh.
\newblock Coloring half-planes and bottomless rectangles.
\newblock {\em Comput. Geom.}, 45(9):495--507, 2012.
\newblock \href {https://doi.org/10.1016/j.comgeo.2011.09.004}
  {\path{doi:10.1016/j.comgeo.2011.09.004}}.

\bibitem{octants}
Bal\'{a}zs Keszegh and D\"{o}m\"{o}t\"{o}r P\'{a}lv\"{o}lgyi.
\newblock Octants are cover-decomposable.
\newblock {\em Discrete Comput. Geom.}, 47(3):598--609, 2012.
\newblock \href {https://doi.org/10.1007/s00454-011-9377-1}
  {\path{doi:10.1007/s00454-011-9377-1}}.

\bibitem{MR3216669}
Bal\'{a}zs Keszegh and D\"{o}m\"{o}t\"{o}r P\'{a}lv\"{o}lgyi.
\newblock Convex polygons are self-coverable.
\newblock {\em Discrete Comput. Geom.}, 51(4):885--895, 2014.
\newblock \href {https://doi.org/10.1007/s00454-014-9582-9}
  {\path{doi:10.1007/s00454-014-9582-9}}.

\bibitem{abafree}
Bal\'{a}zs Keszegh and D\"{o}m\"{o}t\"{o}r P\'{a}lv\"{o}lgyi.
\newblock An abstract approach to polychromatic coloring: shallow hitting sets
  in {ABA}-free hypergraphs and pseudohalfplanes.
\newblock {\em J. Comput. Geom.}, 10(1):1--26, 2019.
\newblock \href {https://doi.org/10.20382/jocg.v10i1a1}
  {\path{doi:10.20382/jocg.v10i1a1}}.

\bibitem{3propercol}
Bal\'{a}zs Keszegh and D\"{o}m\"{o}t\"{o}r P\'{a}lv\"{o}lgyi.
\newblock Proper coloring of geometric hypergraphs.
\newblock {\em Discrete Comput. Geom.}, 62(3):674--689, 2019.
\newblock \href {https://doi.org/10.1007/s00454-019-00096-9}
  {\path{doi:10.1007/s00454-019-00096-9}}.

\bibitem{kovacs}
Istv\'{a}n Kov\'{a}cs.
\newblock Indecomposable coverings with homothetic polygons.
\newblock {\em Discrete Comput. Geom.}, 53(4):817--824, 2015.
\newblock \href {https://doi.org/10.1007/s00454-015-9687-9}
  {\path{doi:10.1007/s00454-015-9687-9}}.


\bibitem{MR76368}
F.~W. Levi.
\newblock \"{U}berdeckung eines {E}ibereiches durch {P}arallelverschiebung
  seines offenen {K}erns.
\newblock {\em Arch. Math. (Basel)}, 6:369--370, 1955.
\newblock \href {https://doi.org/10.1007/BF01900507}
  {\path{doi:10.1007/BF01900507}}.

\bibitem{Pach80}
J\'{a}nos Pach.
\newblock Decomposition of multiple packing and covering.
\newblock {\em Diskrete Geometrie, 2. Kolloq. Math. Inst. Univ. Salzburg},
  pages 169--178, 1980.

\bibitem{Pach86}
J\'{a}nos Pach.
\newblock Covering the plane with convex polygons.
\newblock {\em Discrete Comput. Geom.}, 1(1):73--81, 1986.
\newblock \href {https://doi.org/10.1007/BF02187684}
  {\path{doi:10.1007/BF02187684}}.

\bibitem{unsplittable}
J\'{a}nos Pach and D\"{o}m\"{o}t\"{o}r P\'{a}lv\"{o}lgyi.
\newblock Unsplittable coverings in the plane.
\newblock {\em Adv. Math.}, 302:433--457, 2016.
\newblock \href {https://doi.org/10.1016/j.aim.2016.07.011}
  {\path{doi:10.1016/j.aim.2016.07.011}}.

\bibitem{surveycd}
J\'{a}nos Pach, D\"{o}m\"{o}t\"{o}r P\'{a}lv\"{o}lgyi, and G\'{e}za T\'{o}th.
\newblock Survey on decomposition of multiple coverings.
\newblock In {\em Geometry---intuitive, discrete, and convex}, volume~24 of
  {\em Bolyai Soc. Math. Stud.}, pages 219--257. J\'{a}nos Bolyai Math. Soc.,
  Budapest, 2013.
\newblock \href {https://doi.org/10.1007/978-3-642-41498-5_9}
  {\path{doi:10.1007/978-3-642-41498-5_9}}.

\bibitem{PT08}
J\'{a}nos Pach and G\'{a}bor Tardos.
\newblock Coloring axis-parallel rectangles.
\newblock {\em J. Combin. Theory Ser. A}, 117(6):776--782, 2010.
\newblock \href {https://doi.org/10.1016/j.jcta.2009.04.007}
  {\path{doi:10.1016/j.jcta.2009.04.007}}.

\bibitem{MR2364757}
J\'{a}nos Pach, G\'{a}bor Tardos, and G\'{e}za T\'{o}th.
\newblock Indecomposable coverings.
\newblock In {\em Discrete geometry, combinatorics and graph theory}, volume
  4381 of {\em Lecture Notes in Comput. Sci.}, pages 135--148. Springer,
  Berlin, 2007.
\newblock \href {https://doi.org/10.1007/978-3-540-70666-3_15}
  {\path{doi:10.1007/978-3-540-70666-3_15}}.

\bibitem{MR2679054}
D\"{o}m\"{o}t\"{o}r P\'{a}lv\"{o}lgyi.
\newblock Indecomposable coverings with concave polygons.
\newblock {\em Discrete Comput. Geom.}, 44(3):577--588, 2010.
\newblock \href {https://doi.org/10.1007/s00454-009-9194-y}
  {\path{doi:10.1007/s00454-009-9194-y}}.

\bibitem{PT10}
D\"{o}m\"{o}t\"{o}r P\'{a}lv\"{o}lgyi and G\'{e}za T\'{o}th.
\newblock Convex polygons are cover-decomposable.
\newblock {\em Discrete Comput. Geom.}, 43(3):483--496, 2010.
\newblock \href {https://doi.org/10.1007/s00454-009-9133-y}
  {\path{doi:10.1007/s00454-009-9133-y}}.

\bibitem{MR1689273}
Ioannis Papadoperakis.
\newblock An estimate for the problem of illumination of the boundary of a
  convex body in {$E^3$}.
\newblock {\em Geom. Dedicata}, 75(3):275--285, 1999.
\newblock \href {https://doi.org/10.1023/A:1005056207406}
  {\path{doi:10.1023/A:1005056207406}}.

\bibitem{Prymak}
A~Prymak.
\newblock Every $3 $-dimensional convex body can be covered by $14 $ smaller
  homothetic copies.
\newblock {\em arXiv preprint arXiv:2112.10698}, 2021.

\bibitem{sandssauer}
Bill Sands, Norbert~W. Sauer, and Robert~E. Woodrow.
\newblock On monochromatic paths in edge-coloured digraphs.
\newblock {\em J. Combin. Theory Ser. B}, 33(3):271--275, 1982.
\newblock \href {https://doi.org/10.1016/0095-8956(82)90047-8}
  {\path{doi:10.1016/0095-8956(82)90047-8}}.

\bibitem{MR2844088}
Shakhar Smorodinsky and Yelena Yuditsky.
\newblock Polychromatic coloring for half-planes.
\newblock {\em J. Combin. Theory Ser. A}, 119(1):146--154, 2012.
\newblock \href {https://doi.org/10.1016/j.jcta.2011.07.001}
  {\path{doi:10.1016/j.jcta.2011.07.001}}.

\bibitem{TT07}
G\'{a}bor Tardos and G\'{e}za T\'{o}th.
\newblock Multiple coverings of the plane with triangles.
\newblock {\em Discrete Comput. Geom.}, 38(2):443--450, 2007.
\newblock \href {https://doi.org/10.1007/s00454-007-1345-4}
  {\path{doi:10.1007/s00454-007-1345-4}}.



\end{thebibliography}
\end{document}